%% file: arxiv-main.tex
\documentclass[aap]{arxiv-imsart}
\usepackage{amsthm}
\usepackage{amsmath,amssymb,graphicx}
\usepackage{float}
\RequirePackage[numbers,sort&compress]{natbib}
\RequirePackage[colorlinks,citecolor=blue,urlcolor=blue]{hyperref}
\RequirePackage{graphicx}

\startlocaldefs
\theoremstyle{plain}

\newtheorem{theorem}{Theorem}
\newtheorem{lemma}{Lemma}%
\newtheorem{proposition}{Proposition}%
\newtheorem{corollary}{Corollary}%
\newtheorem{rem}{Remark}%
\RequirePackage[numbers,sort&compress]{natbib}
\usepackage[framemethod=default]{mdframed}
\usepackage{etoc} %
\usepackage{fancyhdr}
\usepackage{nicefrac}

\newmdenv[topline=false,rightline=false,bottomline=false,nobreak=false]{proofaside}

\usepackage{pifont}
\usepackage{times}

\RequirePackage[OT1]{fontenc}
\hypersetup{  urlcolor=blue}
\usepackage{array}
\usepackage{xr}
\usepackage{mathtools}
\usepackage{amsmath,amssymb}
\usepackage{url}
\usepackage{comment}
\usepackage{graphicx}
\usepackage{enumitem}
\usepackage{wrapfig}
\usepackage{framed}
\usepackage{caption}
\usepackage{xspace}
\usepackage{graphicx}
\usepackage{booktabs}

\usepackage[capitalize]{cleveref}  

\let\oldproposition\proposition
\renewcommand{\proposition}{%
  \crefalias{theorem}{nprop}%
  \oldproposition}
\let\oldrem\rem
\renewcommand{\rem}{%
  \crefalias{theorem}{nrem}%
  \oldrem}
\let\oldcorollary\corollary
\renewcommand{\corollary}{%
  \crefalias{theorem}{ncor}%
  \oldcorollary}
\let\oldlemma\lemma
\renewcommand{\lemma}{%
  \crefalias{theorem}{nlem}%
  \oldlemma}
\crefname{nrem}{Rem.}{Rems.}
\crefname{rem}{Rem.}{Rems.}
\crefname{theorem}{Thm.}{Thms.}
\crefname{lemma}{Lem.}{Lems.}
\crefname{nlem}{Lem.}{Lems.}
\crefname{corollary}{Cor.}{Cors.}
\crefname{ncor}{Cor.}{Cors.}
\crefname{proposition}{Prop.}{Props.}
\crefname{nprop}{Prop.}{Props.}
\crefname{assumption}{Assump.}{Assumps.}
\crefname{talign}{}{}
\crefname{section}{Sec.}{Secs.}
\crefname{appendix}{App.}{Apps.}
\crefname{equation}{}{}
\crefformat{footnote}{#2\footnotemark[#1]#3}
\usepackage{mathrsfs}
\usepackage{autonum} 

\input{arxiv-lwm-macros}

\graphicspath{{arxiv-src/figs/},{figs/},{../arxiv-src/figs/}}

\newcommand{\mytitle}{Efficient Concentration with Gaussian Approximation}
\newcommand{\myshorttitle}{Efficient Concentration}
\newcommand{\myabstract}{%
Concentration inequalities for the sample mean, like those due to Bernstein, Hoeffding, and Bentkus, are valid for any sample size but overly conservative, yielding confidence intervals that are unnecessarily wide. The central limit theorem (CLT) provides asymptotic confidence intervals with optimal width, but these are invalid for all sample sizes. To resolve this tension, we develop new computable concentration inequalities for bounded variables with asymptotically optimal size, finite-sample validity, and sub-Gaussian decay. These bounds enable the construction of efficient confidence intervals with correct coverage for any sample size and efficient empirical Berry-Esseen bounds that require no prior knowledge of the population variance. We derive our inequalities by tightly bounding non-uniform Kolmogorov and Wasserstein distances to a Gaussian using zero-bias couplings and Stein's method of exchangeable pairs and demonstrate practical improvements over the Bernstein, Hoeffding, Bentkus, Berry-Esseen, Feller-Cramér, Romano-Wolf, empirical Bernstein, empirical Bentkus, and coin-betting inequalities.}

\endlocaldefs
\begin{document}
\etoctocstyle{1}{Table of contents}
\etocdepthtag.toc{mtchapter}
\etocsettagdepth{mtchapter}{section}
\begin{frontmatter}

\title{\mytitle}
\runtitle{\myshorttitle}

\begin{aug}

\author[A]{\fnms{Morgane} \snm{Austern} \ead[label=e1]{morgane.austern@gmail.com}}
\and
\author[B]{\fnms{Lester} \snm{Mackey} \ead[label=e2]{lmackey@microsoft.com}}

\address[A]{Department of Statistics, Harvard University \printead[presep={ ,\ }]{e1}}%

\address[B]{Microsoft Research New England \printead[presep={ ,\ }]{e2}}
\end{aug}
\begin{abstract}
    \myabstract
\end{abstract}%


\end{frontmatter}

\input{arxiv-sections/intro}
\input{arxiv-sections/defining}

\input{arxiv-sections/zero}
\input{arxiv-sections/wasserstein}
\input{arxiv-sections/empirical}
\input{arxiv-sections/experiments}
\input{arxiv-sections/mc}
\input{arxiv-sections/discussion}

\section*{Acknowledgments}  
We thank Diego Martinez-Taboada and Aaditya Ramdas for providing their empirical Bernstein variance bound code and Aadit Jain, Fred J. Hickernell, Art B. Owen, and Aleksei G. Sorokin for sharing their Monte Carlo confidence interval code.
\appendix
\numberwithin{lemma}{section}
\numberwithin{proposition}{section}
\numberwithin{theorem}{section}
\numberwithin{figure}{section}
\numberwithin{table}{section}
\numberwithin{corollary}{section}
\numberwithin{assumption}{section}

\section*{Table of Contents}
\etocsettocstyle{}{}
    \etocdepthtag.toc{mtappendix}
    \etocsettagdepth{mtchapter}{none}
    \etocsettagdepth{mtappendix}{section}
    \etocsettagdepth{mtappendix}{subsection}
    {\tableofcontents}
\begin{appendix}
\input{arxiv-sections/zero-proof}
\input{arxiv-sections/wasserstein-proof}
\input{arxiv-sections/additional-lemmas}
\end{appendix}

\bibliographystyle{imsart-nameyear} %
\bibliography{refs} 

\end{document}

%% file: arxiv-lwm-macros.tex
\newcommand{\sig}{\sigma}
\newcommand{\eps}{\epsilon}
\newcommand{\lam}{\lambda}

\newcommand{\dee}{d}
\newcommand{\dt}{\,\dee t}

\newcommand{\dx}{\,\dee x}
\newcommand{\dy}{\,\dee y}

\def\balign#1\ealign{\begin{align}#1\end{align}}
\def\baligns#1\ealigns{\begin{align*}#1\end{align*}}
\def\balignat#1\ealign{\begin{alignat}#1\end{alignat}}
\def\balignats#1\ealigns{\begin{alignat*}#1\end{alignat*}}
\def\bitemize#1\eitemize{\begin{itemize}#1\end{itemize}}
\def\benumerate#1\eenumerate{\begin{enumerate}#1\end{enumerate}}

\newenvironment{talign*}
 {\let\displaystyle\textstyle\csname align*\endcsname}
 {\endalign}
\newenvironment{talign}
 {\let\displaystyle\textstyle\csname align\endcsname}
 {\endalign}

\def\balignst#1\ealignst{\begin{talign*}#1\end{talign*}}
\def\balignt#1\ealignt{\begin{talign}#1\end{talign}}
\newcommand{\qtext}[1]{\quad\text{#1}\quad} 
 
\newcommand{\stext}[1]{\ \text{#1}\ } 
\newcommand{\sstext}[1]{\ \ \text{#1}\ \ }

\let\originalleft\left
\let\originalright\right
\renewcommand{\left}{\mathopen{}\mathclose\bgroup\originalleft}
\renewcommand{\right}{\aftergroup\egroup\originalright}

\def\Cramer{Cram\'er\xspace}

\def\tinycitep*#1{{\tiny\citep*{#1}}}
\def\tinycitealt*#1{{\tiny\citealt*{#1}}}
\def\tinycite*#1{{\tiny\cite*{#1}}}
\def\smallcitep*#1{{\scriptsize\citep*{#1}}}
\def\smallcitealt*#1{{\scriptsize\citealt*{#1}}}
\def\smallcite*#1{{\scriptsize\cite*{#1}}}

\def\mbi#1{\boldsymbol{#1}} %

\def\mbb#1{\mathbb{#1}}

\def\mrm#1{\mathrm{#1}}
\def\tbf#1{\textbf{#1}}

\def\reals{\mathbb{R}} %

\def\naturals{\mathbb{N}} %

\def\<{\left\langle} %
\def\>{\right\rangle}

\def\defeq{\triangleq} %
\def\half{\frac{1}{2}}
\def\third{\frac{1}{3}}

\def\norm#1{\|{#1}\|} %
\newcommand{\infnorm}[1]{\norm{#1}_{\infty}} %
\def\indic#1{\mbb{I}\left({#1}\right)} %
\def\E{\mbb{E}} %

\def\P{\mbb{P}} %
\def\Parg#1{\P\left({#1}\right)}

\def\Var{\mrm{Var}} %

\newcommand{\Gsn}{\mathcal{N}}

\newcommand{\Ber}{\textnormal{Bernoulli}}

\newcommand{\Bin}{\textnormal{Binomial}}

\newcommand{\Unif}{\textnormal{Uniform}}

\newcommand{\eqdist}{\stackrel{d}{=}}

\newcommand{\toprob}{\stackrel{p}{\to}}
\newcommand{\toas}{\stackrel{a.s.}{\to}}
\newcommand{\eqas}{\stackrel{a.s.}{=}}
\newcommand{\leas}{\stackrel{a.s.}{\le}}
\newcommand{\geas}{\stackrel{a.s.}{\ge}}

\newcommand{\iid}{\textrm{i.i.d.}\@\xspace}
\newcommand{\dist}{\sim}
\newcommand{\distiid}{\overset{\textrm{\tiny\iid}}{\dist}}

\providecommand{\argmax}{\mathop\mathrm{arg max}} %

\def\support#1{\mathrm{support}({#1})}

\ifdefined\nonewproofenvironments\else
\ifdefined\ispres\else
\newtheorem{assumption}{Assumption}
\fi

\newcommand{\ncref}[1]{\cref{#1}: \nameref*{#1}} %

\newcommand{\pcref}[1]{\NoCaseChange{Proof of \ncref{#1}}} %

\newcommand{\wass}[1][p]{\mathcal{W}_{#1}}
\newcommand{\wbar}{\bar{W}_n}

\newcommand{\onetail}{\hat{Q}_n} %
\newcommand{\twotail}{\hat{Q}^d_n} %
\newcommand{\oneq}{\tilde{q}_n} %
\newcommand{\oneqhat}{\hat{q}_n} %
\newcommand{\twoq}{\tilde{q}_n^d} %
\newcommand{\twoqhat}{\hat{q}_n^d} %
\newcommand{\mnkappa}{M_{n,\kappa}}

\newcommand{\Rsig}[1][\sigma]{\hyperref[Rsig]{\color{black}{R_{#1}}}} %
\newcommand{\tRsig}[1][\sigma]{\widetilde{\Rsig}}
\newcommand{\muhat}{\hat{\mu}}
\newcommand{\mhat}{\hat{m}}

\newcommand{\wassbd}[1][\sigma]{\hyperref[cabris]{\color{black}{\omega^R_p(#1)}}} %
\newcommand{\KRsig}[1][\sigma]{\hyperref[KRsig]{\color{black}{K_{R,\sig}}}} 
\newcommand{\Rsigsqd}[1][\sigma]{\hyperref[Rsig]{\color{black}{R_{#1}^2}}} 
\newcommand{\Rsigpowfour}[1][\sigma]{\hyperref[Rsig]{\color{black}{R_{#1}^4}}}

\newcommand{\ks}{\textup{KS}} %
\newcommand{\kd}{d_{K}} %
\newcommand{\kdone}{d_{K,1}} %

\newcommand{\ksq}[1][\frac{a}{2}]{q_{n,2}^{\ks}(#1)} %
\newcommand{\ksqone}[1][\frac{a}{2}]{q_{n,1}^{\ks}(#1)} %
\newcommand{\empq}{\hat{q}^{d}_n(R,\delta,a_n)}

\newcommand{\vup}[1][]{v_{n#1}}
\newcommand{\vlow}[1][]{\tilde{v}_{n#1}}
\newcommand{\sighat}{\hat{\sig}}

\newcommand{\sigup}{\hat{\sigma}_{\mathrm{up},a}}
\newcommand{\sigupn}{\hat{\sigma}_{\mathrm{up},a_n}}

\newcommand{\siglow}{\hat{\sigma}_{\mathrm{low},a}}
\newcommand{\siglown}{\hat{\sigma}_{\mathrm{low},a_n}}
\newcommand{\sigmid}{\tilde{\sigma}_{n,2}}
\newcommand{\sigmidone}{\tilde{\sigma}_{n,1}}

\newcommand{\Sz}{S^{\star}} %
\newcommand{\Tz}{T^{\star}} %
\newcommand{\Vz}{V^{\star}} %
\newcommand{\Smid}{S'} %
\newcommand{\Tmid}{T'} %
\newcommand{\Vmid}{V'} %
\newcommand{\bsig}[1][n]{\bar\sig_{#1}}

\newcommand{\Nmin}{N_{\min}} 
\newcommand{\Nmint}{\tilde{N}_{\min}} 
\newcommand{\td}{\tilde{d}}

%% file: arxiv-sections/intro.tex
\section{Introduction}
Concentration inequalities for the sample mean are ubiquitous in probability theory, statistics, and machine learning. 
Given $n$ observations from an infinite sequence of independent and identically distributed (\iid) random variables $(W_i)_{i=1}^\infty$, they allow us to give finite-sample and high-probability guarantees that the sample mean $\bar{W}_n\defeq \frac{1}{n}\sum_{i=1}^n W_i$ is close to the population mean $\E[W_1]$. 
More specifically, they provide upper bounds for the probability $\P\big(\bar{W}_n-\E[W_1]> t/\sqrt{n}\big)$ for each $t \ge 0$. 
Such inequalities lie at the heart of decision-making in reinforcement learning \citep{audibert2009exploration}, 
generalization guarantees in high-dimensional statistics, machine learning, and deep learning \citep{wainwright2019high,bartlett2002rademacher,zhou2018non}, and 
the design \citep{maurer2009empirical} and selection \citep{mnih2008empirical} of efficient learning procedures.

However, standard concentration inequalities are overly conservative yielding confidence intervals that are unnecessarily wide and generalization guarantees that are weaker than needed. This is notably the case for the commonly used concentration inequalities of \citet{hoeffding1963probability} and \citet{bernstein1924modification}. 
For bounded random variables $W_i \in [0,R]$ with variance $\sigma^2\defeq \Var(W_i) > 0$, the Hoeffding  and Bernstein inequalities respectively state that the scaled deviation $S_n\defeq \sqrt{n}(\wbar-\E[W_1])$ satisfies
\begin{talign}%
\P(S_n > \sig u)
\le \exp\big({-\frac{2u^2\sig^2}{R^2}}\big)
\ \sstext{and}\  %
\label{eq:bern_tails}
\P(S_n > \sig u)&\le 
\exp\big({-\frac{u^2}{2(1+Ru/( 3\sig\sqrt{n}))}}\big), \ \ 
\ \forall u \geq 0.
\end{talign} 
Meanwhile, the central limit theorem (CLT) identifies the exact limit for each tail probability:
\begin{talign}\label{eq:gaussian_tails}
\P(S_n> \sig u)\xrightarrow{n\rightarrow \infty}\Phi^c(u)
\qtext{for all} u \in \reals,
\end{talign}
where $\Phi$ is the cumulative distribution function (CDF) of a standard normal distribution and  $\Phi^c=1-\Phi$. 
As a result, standard confidence intervals based on the CLT are asymptotically exact and often much narrower than those obtained using concentration inequalities. However, these intervals are typically only asymptotically valid and provide incorrect coverage for every sample size $n$. %

The choice between loose but valid concentration inequalities and tight but invalid CLT-based approximations is very unsatisfying. In this paper we derive new bounds that offer the best of both worlds: our new concentration inequalities are both finite-sample valid and \emph{efficient}---that is, asymptotically of minimal size when scaled by $\sqrt{n}$. 
For example, our primary result, \cref{efficient-zero-tail}, 
implies that, for all $u \geq \delta_n \defeq \frac{R}{\sig\sqrt{n}}$, 
\begin{talign}\label{eq_intro}
\P(S_n > \sigma u) 
    \leq 
\Phi^c(u)
    +
\frac{R}{\sig\sqrt{n}}\varphi(u-\delta_n)
    +
\frac{R\,(4+2u)}{\sig\sqrt{n}}
e^{-\frac{\sig^2}{R^2}(u-\frac{3}{2}\delta_n)_+^2}
\end{talign}
where $\varphi(u)\defeq\frac{1}{\sqrt{2\pi}}e^{-u^2/2}$ is the Lebesgue density of a standard normal.
Underlying this implication is a new, computable concentration inequality formed by explicitly bounding a non-uniform Kolmogorov distance between the sample mean and a Gaussian.

It is informative to compare the result \cref{eq_intro} with those obtained using classical CLT corrections. 
The Berry-Esseen bound \citep{esseen1942liapunov} guarantees that 
\begin{talign}\label{eq:berryesseen}
\P(S_n > \sig u)
\le
\Phi^c(u)
+
\frac{C_{R,\sig}}{\sqrt{n}}
 \qtext{for all} u \geq 0
\end{talign}
and a constant $C_{R,\sig}$ depending only on $R$ and $\sig$. 
This yields an efficient concentration inequality, but the bound is overly conservative as the correction is independent of $u$. Non-uniform Berry-Esseen bounds  \citep{nagaev1965some,bikyalis1966estimates} ameliorate this behavior by  %
identifying a constant $\tilde C_{R,\sig}$ satisfying 
\begin{talign}\label{eq:nonuniformberryesseen}
\P(S_n > \sig u)
\le 
\Phi^c(u)
+
\frac{\tilde C_{R,\sig}}{\sqrt{n}(1+u)^3} \qtext{for all} u \geq 0.
\end{talign}
Appealingly, this non-uniformity yields tighter bounds for larger $u$. 
However, the correction has only cubic, that is, $O(u^{-3})$, decay in $u$ as the underlying argument only exploits the existence of a third moment of $\wbar$. 
Quantile coupling inequalities \citep[see, e.g.,][]{mason2012quantile} like the groundbreaking Komlós-Major-Tusnády approximations \citep[Thm.~1]{komlos1975approximation,komlos1976approximation} and the strong embedding bounds of \citet{chatterjee2012new,bhattacharjee2016strong} improve this $u$ dependence for $W_1$ with finite exponential moments but provide at best exponential decay in $u$ and $O(\log n/\sqrt{n})$ decay in $n$. 
By exploiting the boundedness of $W_1$, our new efficient correction term \cref{eq_intro} guarantees faster,  sub-Gaussian %
$e^{-\Omega(u^2)}$ decay 
and eliminates the extraneous $\log n$ factor present in prior quantile coupling inequalities. 
The aforementioned quantile coupling and strong embedding bounds are also unsuitable for practical deployment due to their unidentified constants. 
Crucially for our applications,  our \cref{efficient-zero-tail} is fully computable, allowing us to develop practical efficient confidence regions in \cref{empi,sec:numerical}.

In \cref{simple}, 
we supplement our primary result with a computational refinement that yields tighter tail and quantile bounds for larger %
deviations by carefully controlling the $p$-Wasserstein distance to Gaussianity. 
For example, our efficient Wasserstein tail bound, \cref{cpt_concentration}, improves the dependence on $u$ at the expense of a worse dependence on $n$ and provides an explicit relative error bound in the spirit of classical \Cramer-type inequalities \citep{cramer1938nouveau,feller1943generalization,petrov1975sums,chen2013stein,fang2023p}:
\begin{talign}\notag
\P(S_n > \sig u) 
    &\leq 
\Phi^c(u - \delta_{u,n}) 
    +
\frac{\varphi(u) }{\sqrt{n}}
    \qtext{for}
\delta_{u,n}\defeq\frac{e \KRsig}{\sqrt{n}} \lceil\log\big(\frac{\sqrt{n}}{\varphi(u)}\big)\rceil
    \\&\leq 
\Phi^c(u)\cdot\big(e^{ {(u+1)\delta_{u,n}}} %
    +
\frac{1}{\sqrt{n}}\frac{\varphi(u)}{\Phi^c(u)}\big)
    \qtext{for all}
u > \delta_{u,n}.
\label{eq:rel-err-wass-bound}
\end{talign} 
Here, $\KRsig$ is an explicit constant depending only on $(R,\sig)$ that we define in \cref{simple}. 
Unlike \cref{cpt_concentration}, most \Cramer-type relative error bounds, including those derived by \citet{cramer1938nouveau}, \citet{petrov1975sums}, \citet{chen2013stein}, and \citet{fang2023p}, are unsuitable for practical use due to unidentified constants. 
A notable exception is the inequality of \citet[Thm.~1]{feller1943generalization} which delivers the \Cramer-type bound
\begin{talign}
\label{eq:feller-cramer}
\P(S_n > \sig u) 
\le 
\Phi^c(u) \exp(\frac{6u^3 c_{R,\sig}}{7\sqrt{n}-84uc_{R,\sig}})\big(1 + \frac{9\sqrt{2\pi}c_{R,\sig}}{\sqrt{n}} \frac{\varphi(u)}{\Phi^c(u)}\big)
 \sstext{for all} u \in (0,\frac{\sqrt{n}}{12c_{R,\sig}})
\end{talign}
for a constant $c_{R,\sig}$ depending only on $(R,\sig)$. 
We will see in \cref{sec:numerical} that \cref{cpt_concentration} provides a substantially tighter bound than the Feller-\Cramer inequality \cref{eq:feller-cramer} in all of our experimental settings.

Another important comparison is with the  ``missing-factor'' bounds of 
\citet{eaton1970note,eaton1974probability},
\citet{talagrand1995missing}, 
Pinelis~\citep{pinelis1994extremal,pinelis2006binomial,pinelis2016optimal},
and Bentkus~\citep{bentkus2001inequality,bentkus2002inequality,bentkus2004hoeffding,bentkus2006domination},
which imply tail decay proportional to $\Phi^c(u) \sim \frac{1}{u}\varphi(u)$.  %
A representative example is the computable bound of \citet[Thm.~2.1]{bentkus2006domination},
\begin{talign}\label{eq:bentkus_tails}
\P(S_n > \sig u)
    \le 
\inf_{t \in [0,u)}
\frac{\E[(G_n - t)_+^2]}{(u - t)^2}
\qtext{for all}
u > 0
\end{talign}
where $G_n \defeq \frac{1}{\sqrt{n}}\sum_{i=1}^n \eps_i$ for \iid $(\eps_i)_{i\ge 1}$ with $\P(\eps_i = \frac{\Rsig}{\sig}) = \frac{\sig^2}{\Rsigsqd + \sig^2} = 1-\P(\eps_i = -\frac{\sig}{\Rsig})$,  
\begin{align+}\label{Rsig}\tag{$R_\sig$}
\textstyle
\Rsig \defeq \half R + \half \sqrt{R^2 - 4\sig^2},
\end{align+}
and $a_+\defeq\max(a,0)$ for $a\in\reals$. 
Like our new inequalities in \cref{efficient-zero-tail,cpt_concentration}, the Bentkus bound \cref{eq:bentkus_tails} eschews a closed-form to obtain a significantly tighter tail bound that is still straightforwardly computable \citep[App.~C]{kuchibhotla2021near}.
However, unlike our new inequalities, the Bentkus bound \cref{eq:bentkus_tails} is  inefficient with a limit at least twice as large as the ideal size $\Phi^c(u)$. 
\begin{proposition}[Inefficiency of Bentkus bound]
For each $u> 0$,
\begin{talign}
\inf_{t \in [0,u)}
\frac{\E[(G_n - t)_+^2]}{(u - t)^2}
   \ \xrightarrow{n\rightarrow \infty}\ 
\inf_{t\in[0,u)} 
\frac{\E[(Z-t)_+^2]}{(u-t)^2}
    \geq 
2\Phi^c(u).
\end{talign}
\end{proposition}
\begin{proof}
Fix any $u> 0$ and any positive $\eps < \E[(Z-u)_+^2]$, %
where $Z$ is a standard Gaussian.
By the $2$-Wasserstein CLT \citep[Thm.~1]{bonis2020stein}, there exists an integer $N$ such that
$\sup_{t\leq u}|\E[(G_n-t)_+^2] -\E[(Z-t)_+^2]| \leq \eps$
for all $n \geq N$. Now, fix any $n \geq N$ and suppose
\begin{talign}
f_n(t) \defeq \E[(G_n-t)_+^2]/(u-t)^2 \leq f_n(0) = \E[(G_n)_+^2]/u^2 \leq 1/u^2
\end{talign}
for some $t\in[0,u)$.  
Then we necessarily have
\begin{talign}
{(\E[(Z-u)_+^2]-\eps)}{/(u-t)^2}
    \leq
{(\E[(Z-t)_+^2]-\eps)}{/(u-t)^2}
    \leq
f_n(t)
    \leq
1/u^2.
\end{talign}
Parallel logic ensures $f(t) \defeq \frac{\E[(Z-t)_+^2]}{(u-t)^2} \leq f(0)$ only if $\frac{1}{(u-t)^2}\leq \frac{1}{u^2 \E[(Z-u)_+^2]}$. Therefore,
\begin{talign}
|\inf_{t\in[0,u)} f_n(t) - \inf_{t\in[0,u)} f(t)|
    \leq 
\frac{\eps}{u^2(\E[(Z-u)_+^2]-\eps)}.
\end{talign}
The result now follows from the arbitrariness of $\eps$ and the inequality of \citep[Thm.~7.1]{bentkus2006domination}.
\end{proof}

In \cref{empi}, we apply our new tools to develop practical \emph{empirical Berry-Esseen bounds} that are efficient and finite-sample valid even when the variance parameter $\sigma^2$ is unknown.
In \cref{sec:numerical}, we confirm numerically that our new bounds yield improvements over the Hoeffding, Bernstein, and Bentkus inequalities, the uniform and non-uniform Berry-Esseen corrections, the generalized \Cramer bound of Feller, the popular but inefficient empirical Bernstein \citep{mnih2008empirical,audibert2009exploration,maurer2009empirical} and empirical Bentkus \citep{kuchibhotla2021near} bounds, and the efficient $I_{n,3}$ confidence interval of \citet{romano2000finite}. 
We conclude with an application to Monte Carlo integration (\cref{sec:mc}) in which our efficient empirical Berry-Esseen bounds yield narrower confidence intervals than the state-of-the-art predictable plug-in and betting intervals of \citet{waudby2024estimating} for larger sample sizes. 
\cref{sec:discussion} presents a discussion of these results and related work.

%% file: arxiv-sections/defining.tex
\section{Defining Efficient Concentration}\label{def}
To match the setting of the classical Hoeffding and Bernstein inequalities \cref{eq:bern_tails}, we will focus on random variables satisfying the following distributional assumptions.
\def\theassumption{$(R,\sigma)$}
\addtocounter{assumption}{-1}
\begin{assumption}\label{assump:bounded-dev}
The scaled deviations $(S_n)_{n\geq 1}$ satisfy $S_n = \sqrt{n} (\wbar - \E[W_1])$ for %
\iid variables $(W_i)_{i=1}^\infty$ with $\wbar\defeq \frac{1}{n}\sum_{i=1}^n W_i$, $\Var(W_1) = \sigma^2 > 0$,  and $W_1 \in [0, R]$ almost surely.
\end{assumption}

 Our first inferential goal is to tightly upper bound the tail probability $\P(S_n > \sig u)$ for a given threshold $u \geq 0$.
The CLT provides an asymptotic lower bound for this problem as $\P(S_n > \sig u)$ is known to converge precisely to $\Phi^c(u)$ as $n$ increases. %
\begin{proposition}[Asymptotic lower bound for valid tail bounds]
\label{tail_lower_bound}
Fix any $u \geq 0$ 
and any sequence of candidate tail bounds $(\delta_n(u))_{n\geq 1}$. 
Under \cref{assump:bounded-dev}, if $\P(S_n > \sig u) \leq \delta_n(u)$ for all $n$, then $\Phi^c(u) \leq \liminf_{n\to\infty} \delta_n(u).$
\end{proposition}
\begin{proof}
Suppose that a sequence $\delta_n(u)$ satisfies  $\liminf_{n\to\infty} \delta_n(u) < \Phi^c(u)$.  Then there exists an $\eps > 0$ such that $\delta_n(u) \leq \Phi^c(u) - \eps$ for infinitely many $n$.
However, by the CLT \citep[see, e.g.,][Thm.~3.4.1]{durrett2019probability}, there exists an $n_{\eps}$ such that, for all $n > n_{\eps}$, $\Phi^c(u) - \eps < \P(S_n > \sig u)$.  Therefore, $\delta_n(u) < \P(S_n > \sig u)$ for infinitely many $n$, confirming the claim via its contrapositive.
\end{proof}

Unfortunately, the CLT limit does not provide a suitable tail bound for any finite $n$. %
However, by tightly bounding the distance between the distribution of $S_n$ and the distribution of a Gaussian we can correct the asymptotic bound to obtain one that is both valid in finite samples and asymptotically exact. %
We will call such bounds \emph{efficient concentration inequalities}. 

Our second inferential goal is to tightly bound the quantiles of $S_n$.
That is, given a tail probability $\delta \in (0,1)$ we wish to find $q_n(R,\delta,\sigma)$ (a measurable function of $(W_i)_{i=1}^n, R,\sigma,$ and $\delta$) satisfying $\P(S_n > q_n(R,\delta,\sigma))\le \delta$. Such quantile bounds immediately deliver both one- and two-sided confidence intervals for the population mean $\E[W_1]$ as 
\begin{talign}
\P(\wbar - \frac{q_n(R,\delta,\sigma)}{\sqrt{n}} \leq \E[W_1]) 
\ \wedge \ 
\P(|\E[W_1] - \wbar| \leq \frac{q_n(R,\delta/2,\sigma)}{\sqrt{n}}) \geq 1-\delta.
\end{talign}
The interval efficiency theory of  
 \citet{romano2000finite} implies that the CLT once again provides an asymptotic lower bound for any valid sequence of quantile bounds. %
\begin{proposition}[Asymptotic lower bound for valid quantile bounds]
\label{quantile_lower_bound}
Fix any $R,\sigma > 0$, any $\delta \in (0,\half)$, and any nonnegative candidate quantile bounds $(q_n(R,\delta,\sigma))_{n\geq 1, \sigma >0}$.
Under \cref{assump:bounded-dev}, if $\P(S_n > q_n(R,\delta,\sigma)) \leq \delta$
for all $n$, 
then $(q_n(R,\delta,\sigma))_{n\geq 1}$ is not asymptotically concentrated\footnote{A sequence of nonnegative random variables $(X_n)_{n\geq 1}$ is asymptotically concentrated on $[0,a]$ if $(X_n - a)_+ \toprob 0$.} on $[0,a]$ for any $a < \sigma \Phi^{-1}(1-\delta)$. Here,  $\Phi^{-1}$ is the quantile function of a standard normal distribution.
\end{proposition}
\begin{proof}
The result follows by applying Thm.~2.1 of  \citet{romano2000finite} to the conservative  confidence intervals $I_n(\delta, \sigma) = \bar{W}_n + \frac{1}{\sqrt{n}}[-q_n(R,\delta,\sigma), q_n(R,\delta, \sigma)]$ for the unknown mean $\E[W_1]$.
\end{proof}

To construct \emph{efficient quantile bounds}, we will once again tightly bound the distance between $S_n$ and its Gaussian limit. 

\subsection{Notation} 
Hereafter, we will write $a_n\lessapprox b_n$ to indicate that two sequences $(a_n)_{n\ge1}$ and $(b_n)_{n\ge1}$ satisfy $a_n\le b_n + o_n(b_n)$ and use the shorthand
$\infnorm{X} = \sup \support{|X|}$ and $\norm{X}_p\defeq\E[|X|^p]^{1/p}$ when a random variable $X$ is bounded or has a $p$-th absolute moment for some $p\ge 1$. 
We will also make regular use of the parameter $\Rsig$ 
which provides a priori bounds on the summand mean $\E[W_1]$ and deviation $|W_1 - \E[W_1]|$:

 \begin{lemma}[Summand mean and deviation bounds]
 \label{summand-deviation-bound}
 Under \cref{assump:bounded-dev}, 
 we have 
 \begin{talign}
    R-\Rsig \leq \E[W_1] \leq \Rsig
    \ \stext{and} \ 
    |W_1 - \E[W_1]| \leq \Rsig \ \stext{almost surely}. %
 \end{talign}
 \end{lemma}
 \begin{proof}
  Since $W_1\in[0,R]$, we have
  $
    \E[W_1]^2 = \E[W_1^2] - \sig^2 \leq R \E[W_1] - \sig^2.
  $
  Hence $\E[W_1]$ must lie between the roots $\half R \pm \half \sqrt{R^2-4\sig^2}$ of this quadratic inequality.
 \end{proof}

%% file: arxiv-sections/zero.tex
\section{Efficient Concentration with Zero-Bias Couplings}\label{zero}
To derive our initial efficient concentration inequality, we tightly bound a non-uniform Kolmogorov distance between the scaled deviation $S_n$ and a Gaussian using \emph{zero-bias couplings}.
As in \citet[Prop.~2.1]{chen2011normal}, we say that $\Sz$ has the \emph{zero-bias distribution} for a mean-zero random variable $S$ with $\sigma^2:=\Var(S) <\infty$ if the distribution of $\Sz$ is absolutely continuous with Lebesgue density $p^\star(x) = \E[S\, \indic{S > x}]/\Var(S)$ or, equivalently, if $\sig^2\E[f'(\Sz)] = \E[Sf(S)]$ for all absolutely continuous $f$ with $\E[|S f(S)|] <\infty$.
Our primary result, proved in \cref{proof-efficient-zero-tail}, uses a close coupling of $S_n$ and its zero-biased counterpart $\Sz_n$ to establish efficient concentration. 
More precisely, we construct $\Sz_n$ satisfying $|S_n-\Sz_n|\leq \frac{R}{\sqrt{n}}$ almost surely, show that the tails of $\Sz_n$ are closely bounded by those of a Gaussian, and transfer the bounds to $S_n$ via the small perturbation $\frac{R}{\sqrt{n}}$. 

\begin{theorem}[Efficient zero-bias tail bounds]\label{efficient-zero-tail}
\!Under \cref{assump:bounded-dev}, for all $u\!\geq\! 0$ and $\lam\!\in\![0,1]$,
\begin{talign}
\P(S_n > \sig u \!+\! \frac{R}{\sqrt{n}}) 
    \leq
\Phi^c(u) 
    +
\frac{R}{\sig \sqrt{n}+R\,u}
\big[h_u(\lam u) \!-\! u\Phi^c(u)
    +
(h_u(u) \!-\! h_u(\lam u))\,Q_{n}(\lam \sigma u)\big],
\end{talign}
where, for $\Delta_n \defeq \frac{\Rsig}{4\sqrt{n}}$ and all $w\leq u$,
\begin{talign}
\label{hu-def}
h_u(w) 
    &\defeq 
(w + 
\frac{(1+w^2)\Phi(w)}{\varphi(w)})\Phi^c(u) 
    \qtext{and}\\
\label{Qnu-def}
Q_{n}(u)
    &\defeq 
\min\!\big(
e^{-\frac{2(u-\Delta_n)_+^2}{R^2}}
    ,
e^{-\frac{(u-\Delta_n)_+^2}{2(\vup^2+{\Rsig (u-\Delta_n)/}{(3\sqrt{n})})}}
    ,
\Phi^c(\frac{u-\Delta_n}{\vlow})+ \frac{0.56}{\sqrt{n}}\frac{\Rsig \vup^2\, + \beta_n}{\vlow^3}\big),
\end{talign} 
for $\vup^2\defeq \sig^2 + \frac{1}{9n}(\Rsigsqd-6\sig^2)$,
$\vlow^2 
    \!\defeq
\sig^2(1-\frac{89}{144 n})$, 
and
$\beta_n\!\defeq {\min(\frac{1}{4}\Rsig, \Rsig[\frac{\sqrt{55}\sig}{12}]\!-\!\Rsig)} (\frac{\sig^2}{3n} + \frac{\Rsigsqd}{9n})$.
Moreover, under the same conditions, 
\begin{talign}
\P(&S_n > \sig u\frac{\sqrt{n+1}}{\sqrt{n}} +\frac{\Rsig}{\sqrt{n}}) 
    \\ \label{alt-zero-tail}
    &\leq 
\Phi^c(u) 
    +
\frac{R}{\sig \sqrt{n+1}+R\,u}
\big[h_u(\lam u) - u\Phi^c(u) + (h_u(u) - h_u(\lam u))\,Q_{n+1}(\lam \sigma u\frac{\sqrt{n}}{\sqrt{n+1}})\big].
\end{talign}
\end{theorem}

Like the popular concentration inequalities due to Hoeffding \cref{eq:bern_tails}, Bernstein \cref{eq:bern_tails}, and Bentkus \cref{eq:bentkus_tails}, \cref{efficient-zero-tail} is valid and computable for any sample size $n$. 
However, unlike the Hoeffding, Bernstein, and Bentkus inequalities, \cref{efficient-zero-tail} is also efficient and converges to the asymptotically exact Gaussian tail bound at a $O(1/\sqrt{n})$ rate.
In fact, as our next corollary demonstrates, the suboptimality of \cref{efficient-zero-tail} also decays at a sub-Gaussian $e^{-\Omega(u^2)}$ rate in $u$, faster than the more conservative Berry-Esseen \cref{eq:berryesseen}, non-uniform Berry-Esseen \cref{eq:nonuniformberryesseen}, and quantile coupling inequalities \citep{komlos1975approximation,komlos1976approximation,bretagnolle1989hungarian,chatterjee2012new,mason2012quantile,bhattacharjee2016strong}.

\begin{corollary}[Efficient sub-Gaussian tail bound]\label{subgaussian}
Under \cref{assump:bounded-dev}, 
\begin{talign}
\P(S_n > \sigma u) 
    \leq 
\Phi^c(u)
    +
\frac{R}{\sig\sqrt{n}}\varphi(u-\delta_n)
    +
\frac{R\,(4+2u)}{\sig\sqrt{n}}
e^{-\frac{\sig^2}{R^2}(u-\frac{3}{2}\delta_n)_+^2}
\end{talign}
for all $u \geq \delta_n \defeq \frac{R}{\sig\sqrt{n}}$.
\end{corollary}
\begin{proof}
Fix any $u \ge \delta_n$,
introduce the shifted value $v = u - \delta_n$, and instantiate the notation of \cref{efficient-zero-tail}. 
For each $w\geq 0$, we have  
$Q_n(w) \leq e^{-{2(w-\Delta_n)_+^2/}{R^2}}$ 
by \cref{Qnu-def} and 
\begin{talign}
a_v(w)
    \defeq
\frac{(1+w^2)\Phi(w)}{\varphi(w)}\Phi^c(v) 
    \leq
\frac{2(1+w^2)\Phi(w)}{v + \sqrt{v^2 + 8/\pi}}e^{-\frac{v^2-w^2}{2}}
    \leq
(2+w)\Phi(w)e^{-\frac{v^2-w^2}{2}}
\end{talign}
by
\citet[7.1.13]{abramowitz1964handbook}. 
Invoking 
\cref{efficient-zero-tail} with $\lam=1/\sqrt{2}$ therefore yields 
\begin{talign}
\P(&S_n > \sig u) %
    -
\Phi^c(v) \\
    &\leq
\frac{R}{\sig \sqrt{n}+R\,v}
\big[(a_v(\lam v) 
    -
(1-\lam)v\Phi^c(v))(1-Q_{n}(\lam \sigma v))
    +
a_v(v)\,Q_{n}(\lam \sigma v)\big]\\
    &\leq
\frac{R\,(2+v)\Phi(v)}{\sig \sqrt{n}+R\,v}
    \big(
e^{-\frac{(1-\lam^2) v^2}{2}}
    +
e^{-\frac{2(\lam \sig v-\Delta_n)_+^2}{R^2}}
    \big) 
    =
\frac{R\,(2+v)\Phi(v)}{\sig \sqrt{n}+R\,v}
    \big(
e^{-\frac{v^2}{4}}
    +
e^{-\frac{(\sig v-\sqrt{2}\Delta_n)_+^2}{R^2}}
    \big) \\
    &\leq
\frac{2R\,(2+v)\Phi(v)}{\sig \sqrt{n}+R\,v}
e^{-\frac{(\sig v-\sqrt{2}\Delta_n)_+^2}{R^2}} 
    \leq
\frac{2R\,(2+u)}{\sig \sqrt{n}}
e^{-\frac{\sig^2}{R^2}(u-\frac{3}{2}\delta_n)_+^2}.
\end{talign}
The advertised result now follows from the relation
\begin{talign}
\Phi^c(v)
    =
\Phi^c(u)
    +
\int_v^u \varphi(x) dx
    \leq
\Phi^c(u)
    +
(u-v) \varphi(v)
    =
\Phi^c(u)
    +
\delta_n\varphi(u-\delta_n).
\end{talign}
\end{proof}

While the bound in \cref{efficient-zero-tail} is more complex than that in \cref{subgaussian}, it is straightforward to compute\footnote{\label{github}See \url{https://github.com/lmackey/gauss_conc} for our open-source Python implementation.} and significantly tighter in practice. 
In \cref{empi,sec:numerical}, we will use our efficient concentration inequalities to develop efficient empirical Berry-Esseen bounds and efficient confidence intervals for Monte Carlo integration. %
\subsection{\pcref{efficient-zero-tail}}\label{proof-efficient-zero-tail}
The proof of \cref{efficient-zero-tail} relies on four auxiliary results.
The first, proved in \cref{proof-efficient-zero-tail-unidentical}, 
provides a tail bound for a sum, $T_n$, of independent random variables in terms of a sum, $\Tmid_n$, of interpolated zero-biased variables.
\begin{theorem}[Unidentical zero-bias tail bound]\label{efficient-zero-tail-unidentical}
Suppose independent $(V_i)_{i\geq 1}$ satisfy 
$\E[V_i] = 0,$
$\Var(V_i) = \sig_i^2,$ 
and $\sup \support{V_i} - \inf \support{V_i} \leq R$ for all $i\in\naturals$.
For each $n\in\naturals$, 
define the scaled sum 
$T_n = \frac{1}{\sqrt{n}}\sum_{i=1}^n V_i$ 
with
variance parameter $\bar\sig_n^2=\frac{1}{n}\sum_{i=1}^n \sig_i^2$ 
and 
auxiliary variable
\begin{talign}\label{Smidn-def}
\Tmid_n \defeq T_n + \frac{1}{\sqrt{n}}(\Vmid_{I_n} - V_{I_n})
    \ \stext{for}\ 
\Vmid_i \defeq \Vz_i + U(V_i-\Vz_i)
   \ \stext{and}\ 
\P(I_n = i) = \frac{1}{n}\frac{\sig_i^2}{\bsig^2}\,\indic{1\le i\le n},
\end{talign}
where each $\Vz_i$ has the 
zero-bias distribution of $V_i$, and 
$U\sim \Unif([0,1])$, 
$(I_n)_{n\geq 1}$,
and $(\Vz_i)_{i\geq 1}$ are mutually independent and independent of $(V_i)_{i\geq 1}$.
Then, for all $u\geq 0$ and $\lam\in[0,1]$,
\begin{talign}
\P(T_n > \bsig u + \frac{R}{\sqrt{n}}) 
    &\leq
\frac{\bsig \sqrt{n}}{\bsig \sqrt{n}+R\,u}\Phi^c(u) 
  +
\, 
\frac{R}{\bsig \sqrt{n}+R\,u}
(h_u(u) - h_u(\lam u))\,  \P(\Tmid_n > \lam \bsig u)
\\&+\frac{R}{\bsig \sqrt{n}+R\,u}
h_u(\lam u).
\end{talign}
\end{theorem}

The second result, also proved in \cref{proof-efficient-zero-tail-unidentical}, 
provides an alternative tail bound when the summands underlying $T_n$ are \iid
\begin{theorem}[Identical zero-bias tail bound]\label{efficient-zero-tail-identical}
Instantiate the notation and assumptions of \cref{efficient-zero-tail-unidentical}. 
If $(V_i)_{i=1}^{n+1}$ are \iid with $\sig\defeq\sig_1$ and $V_1\geq -R'$ almost surely, then
\begin{talign}
&\P(T_n > \sig u\frac{\sqrt{n+1}}{\sqrt{n}} +\frac{R'}{\sqrt{n}}) \\
    &\leq 
\frac{\sig \sqrt{n+1}}{\sig \sqrt{n+1}+R\,u}\Phi^c(u) 
    +
\, 
\frac{R}{\sig \sqrt{n+1}+R\,u}
(h_u(u) - h_u(\lam u)) 
\P(\Tmid_{n+1} > \lam \sig u\frac{\sqrt{n}}{\sqrt{n+1}})
\\&+\frac{R}{\sig \sqrt{n+1}+R\,u}
h_u(\lam u).
\end{talign}
\end{theorem}

The third result %
bounds the moments of interpolated zero-biased variables.
\begin{lemma}[Properties of $\Vmid_i$]\label{Vmid-properties}
Instantiate the notation and assumptions of \cref{efficient-zero-tail-unidentical}.
For any $i\in\naturals$,
if $|V_i| \leq R'$ almost surely, then 
$\sup \support{\Vmid_i} - \inf \support{\Vmid_i} \leq R,$
\begin{talign}
&|\E[\Vmid_i]| \leq \frac{R'}{4},\quad
\E[{\Vmid_i}^2] \leq \frac{\sig_i^2}{3} + \frac{{R'}^2}{9},\quad
\Var(\Vmid_i) \geq \frac{55\sig_i^2}{144},\\
&|\Vmid_i| \leq R',\qtext{and}
|\Vmid_i - \E[\Vmid_i]| \leq R_i \defeq \min(\half R + \half \sqrt{R^2 - \frac{220}{144}\sig_i^2}, \frac{5}{4}R')
\qtext{almost surely.}
\end{talign} 
\end{lemma}
\begin{proof}
Fix any $i\in\{1,\dots,n\}$, and suppose $|V_i| \leq R'$ almost surely.  
We invoke the definition of $\Vmid_i$, Lem.~2.1(iv) of \citet{goldstein1997stein}, and our boundedness assumption in turn to find that 
\begin{talign}
|\E[{\Vmid_i}]|
    &=
\half |\E[\Vz_i]|
    = 
\frac{1}{4\sigma_i^2}|\E[V_i^3]|
    \leq
\frac{R'}{4}.
\end{talign}
The same invocations, coupled with the independence of $V_i, U, $ and $\Vz_i$, imply
\begin{talign}
\E[{\Vmid_i}^2]
    &=
\E[U^2]\E[V_i^2] + \E[(1-U)^2]\E[{\Vz_i}^2]  
    =
\third \sigma_i^2 + \frac{1}{9\sigma_i^2} \E[{V_i}^4] 
    \leq
\third \sigma_i^2 + \frac{1}{9}{R'}^2.
\end{talign} 
The same upper bound holds for $\Var({\Vmid_i}) = \E[{\Vmid_i}^2] - \E[{\Vmid_i}]^2$. 
Moreover, since $\E[V_i^3]^2 \leq \E[V_i^2]\E[V_i^4]$ by Cauchy-Schwarz and $\E[V_i^2]^2 \leq \E[V_i^4]$ by Jensen's inequality, we have
\begin{talign}
\Var({\Vmid_i})
    &\geq
\third \sigma_i^2
    +
\frac{1}{9\sigma_i^2} \E[{V_i}^4] 
    -
\frac{1}{16\sigma_i^2} \E[{V_i}^4]
    =
\third \sigma_i^2
    +
\frac{7}{144\sigma_i^2} \E[{V_i}^4] 
    \geq
\third \sigma_i^2
    +
\frac{7}{144}\sigma_i^2 
    =
\frac{55}{144}\sigma_i^2.
\end{talign}
Next, by \citet[Lem.~2.1(ii)]{goldstein1997stein}, the support of $\Vz_i$ is the closed convex hull of the support of $V_i$.
Therefore, $|\Vmid_i| \leq R'$ almost surely by the triangle inequality, and 
\begin{talign}
\sup \support{\Vmid_i} - \inf \support{\Vmid_i}
    =
\sup \support{V_i} - \inf \support{V_i}
    \leq
R.
\end{talign}
The second almost sure bound on $|\Vmid_i-\E[\Vmid_i]|$ now follows from the triangle inequality, while the first follows from \cref{summand-deviation-bound} and our lower estimate for $\Var(\Vmid_i)$.
\end{proof}

The fourth result bounds the tail probabilities of the interpolated zero-biased sum $\Tmid_n$ using concentration inequalities for sums of potentially unidentically distributed random variables.
\begin{lemma}[$\Tmid_n$ tail bound]\label{Tmid-tail}
Under the notation and assumptions of \cref{Vmid-properties},
for each $t\in\reals$,
\begin{talign}
&\P(\Tmid_n > t+\frac{R'}{4\sqrt{n}}) 
    \leq 
\frac{1}{n}\sum_{i=1}^n \frac{\sig_i^2}{\bsig^2}\,\P(\frac{1}{\sqrt{n}}(\Vmid_i-\E[\Vmid_i]) + \frac{1}{\sqrt{n}}\sum_{j\neq i} V_j  > t) \\
    &\leq 
\frac{1}{n}\sum_{i=1}^n \frac{\sig_i^2}{\bsig^2}\min\!\big(
e^{-\frac{2t^2}{R^2}}
    ,
e^{-\frac{t^2}{2(\vup[,i]^2+\frac{R't}{3\sqrt{n}})}}
    ,
\Phi^c(\frac{t}{\vlow[,i]})+ \frac{0.56}{\sqrt{n}}\frac{R' \vup[,i]^2\, + \frac{R_i-R'}{n} (\frac{\sig_i^2}{3} + \frac{{R'}^2}{9})}{\vlow[,i]^3}\big)
\end{talign} 
for  
$\vup[,i]^2\defeq \bsig^2 + ({R'}^2-6\sigma_i^2)/(9n)$
and
$\vlow[,i]^2 
    \defeq
\bsig^2-{89 \sig_i^2/}{(144 n)}$.
\end{lemma}
\begin{proof}
The first claim follows from the definition $\Tmid_n = \frac{1}{\sqrt{n}}\Vmid_{I_n} + \frac{1}{\sqrt{n}}\sum_{j\neq I_n} V_j$ and the \cref{Vmid-properties} bound, $\E[\Vmid_i]\leq\frac{R'}{4}$ for each index $i$.
Since $\E[|\Vmid_i-\E[\Vmid_i]|^3]\leq R_i \Var(\Vmid_i)$ and $\E[|V_i|^3] \leq R' \sig_i^2$ for each $i\in\naturals$, the second follows by combining the upper and lower estimates of \cref{Vmid-properties} with the \citet[Thm.~1]{hoeffding1963probability}, Bernstein \citep[Cor.~2.11]{boucheron2013concentration}, and Berry-Esseen \citep[Thm.~1]{shevtsova2010improvement} tail bounds for unidentically distributed summands.
\end{proof}

Under \cref{assump:bounded-dev}, the first result of \cref{efficient-zero-tail} now follows from a direct application of \cref{efficient-zero-tail-unidentical,Tmid-tail} with \iid $V_i = W_i - \E[W_i]$, $T_n = S_n$, $\sig_i = \sigma$, and, by \cref{summand-deviation-bound}, $R' = \Rsig$.
Analogously, under \cref{assump:bounded-dev}, 
the bound \cref{alt-zero-tail} follows from \cref{efficient-zero-tail-identical,Tmid-tail}.

\subsection{Proofs of \cref{efficient-zero-tail-unidentical,efficient-zero-tail-identical}: Unidentical and identical zero-bias tail bounds}\label{proof-efficient-zero-tail-unidentical}
The crux of our argument, inspired by \citet[Thm.~3.27]{ross2011fundamentals} and proved in \cref{proof-zero-tail}, shows that a zero-biased version $\Sz$ of a random variable $S$ has tails close to those of a Gaussian whenever $\Sz$ and $S$ have a close coupling.
\begin{theorem}[General zero-bias tail bounds]\label{zero-tail}
Suppose that $\Sz$ has the zero-bias distribution for a random variable $S$ with $\E[S]=0$ and $\Var(S) = 1$ and that $\Sz-S \leq \delta$ almost surely. 
For $U\sim\Unif([0,1])$ independent of $(S,\Sz)$, 
define the intermediate variable 
$%
\Smid 
    \defeq 
\Sz + U(S-\Sz).
$ %
Then, for all $u\geq 0$ and $\lam\in[0,1]$, 
\begin{talign}
\P(\Sz > u) - \Phi^c(u)
    &\leq
\delta\, 
[h_u(\lam u)
    + 
(h_u(u) - h_u(\lam u)) \P(\Smid > \lam u)
	-
u\, \P(\Smid > u)],
\end{talign}
for $h_u$ defined in \cref{hu-def}.
If, in addition, $S-\Sz \leq \delta$ almost surely,
\begin{talign}
&\P(S > u+\delta)
    \leq 
\frac{1}{1+\delta\,u}\Phi^c(u) 
    +
\, 
\frac{\delta}{1+\delta\,u}
\big(h_u(\lam u)
    +
(h_u(u) - h_u(\lam u)) \P(\Smid > \lam u)\big).
\end{talign}
\end{theorem}

Hence, to establish \cref{efficient-zero-tail-unidentical}, it only remains to construct a suitable zero-bias coupling for $T_n/\bsig$.
By \citet[Lem.~2.8]{chen2011normal}, $\Tz_n = T_n + \frac{1}{\sqrt{n}}( \Vz_{I_n} - V_{I_n})$ has the zero-bias distribution for $T_n$. 
Therefore, by the zero-bias definition, $\Tz_n/\bsig$ has the zero-bias distribution for $T_n/\bsig$.
Furthermore, by \citet[Sec.~2.3.3]{chen2011normal}, $\support{\Vz_i} = [\inf\support{V_i},  \sup\support{V_i}]$ for each $i$, so $|\Tz_n/\bsig - T_n/\bsig| = |\Vz_I - V_I|/(\bsig\sqrt{n}) \leq R/(\bsig\sqrt{n})$ almost surely.
The advertised result thus follows from the second claim of \cref{zero-tail} applied with $S = T_n/\bsig$, $\Sz = \Tz_n/\bsig$, and $\delta = R/(\bsig\sqrt{n})$.

Now suppose $(V_i)_{i=1}^{n+1}$ are \iid with with $\sig\defeq\sig_1$ and $V_1\geq -R'$ almost surely.
By \citet[Lem.~2.8]{chen2011normal}, $\Tz_{n+1} = T_{n+1} +\frac{1}{\sqrt{n+1}}( \Vz_{I_{n+1}} - V_{I_{n+1}})$ has the zero-bias distribution for $T_{n+1}$.
By symmetry, $\sqrt{n}\, T_n + \Vz_{n+1} \eqdist \sqrt{n+1}\,\Tz_{n+1}$ and 
$\sqrt{n}\,T_n + \Vmid_{n+1} \eqdist \sqrt{n+1}\,\Tmid_{n+1}$.
Since $\Vz_{n+1} \wedge \Vmid_{n+1}\geq -R'$ almost surely by \citet[Sec.~2.3.3]{chen2011normal}, 
\begin{talign}&
\P(T_n > \sig u\frac{\sqrt{n+1}}{\sqrt{n}} +\frac{R'}{\sqrt{n}})
    \leq
\P(T_n + \frac{1}{\sqrt{n}}(\Vz_{n+1} \wedge \Vmid_{n+1}) > \sig u\frac{\sqrt{n+1}}{\sqrt{n}})
  \\&  \leq
\P(\Tz_{n+1} > \sig u)
    \wedge
\P(\Tmid_{n+1} > \sig u).
\end{talign}
\cref{efficient-zero-tail-identical} now follows from the first result of \cref{zero-tail} with $S = T_{n+1}/\sig$, $\Sz = \Tz_{n+1}/\sig$, and $\delta = R/(\sig\sqrt{n+1})$.

%% file: arxiv-sections/wasserstein.tex
\section{Computational Refinement  
with Wasserstein Approximation}\label{simple}
To obtain a practical refinement of our initial concentration inequality (\cref{efficient-zero-tail}), we 
next bound a second, complementary distance to Gaussianity, 
the \emph{$p$-Wasserstein distance}, 
\begin{talign}
\wass (S_n,\Gsn(0,\sig^2)) \defeq \inf_{\substack{\gamma\in \Gamma(S_n, \Gsn(0,\sig^2))}}\mathbb{E}_{(X,Y)\sim \gamma}\big(|X-Y|^p\big)^{1/p}
\qtext{for}
p \geq 1,
\end{talign}
where $\Gamma(S_n,\Gsn(0,\sig^2))$ is the set of all couplings between the law of $S_n$ and the normal distribution $\Gsn(0,\sig^2)$.
In \cref{last_night}, we explicitly bound each $p$-Wasserstein distance to Gaussianity to obtain the following efficient and computable concentration inequalities.

\begin{theorem}[Efficient Wasserstein tail bound] \label{cpt_concentration}
Under \cref{assump:bounded-dev}, for any $u\geq 0$ and  auxiliary upper bounds $\onetail(u)$ and $\twotail(u)$ on $P\big(S_n> \sigma u\big)$ and $P\big(|S_n|> \sigma u\big)$ respectively, we have
\begin{talign}\label{eq:wass-tail-bound}
    \P\big(S_n> \sigma u\big)&\le\min\Big(\inf_{\substack{p\in\naturals,\, \rho\in (0,1)}} \Phi^c({\rho u}{}) +\frac{\wassbd^p}{ (1-\rho)^pu^p+\wassbd^p\,\indic{p=2}}\,,\,\onetail(u)\Big) \ \ \text{and}\\
    \P\big(|S_n|> \sigma u\big)&\le\min\Big(\inf_{\substack{p\in\naturals,\, \rho\in (0,1)}} 2\Phi^c({\rho u}{}) +\frac{\wassbd^p}{ (1-\rho)^pu^p}\,,\,\twotail(u)\Big) 
\end{talign}
where $\wassbd$ is a computable bound on $\frac{1}{\sigma}\wass (S_n, \Gsn(0,\sig^2))$ defined in  \cref{lemm:dino_cpt} 
and 
satisfying 
$\wassbd \leq \KRsig\frac{p}{\sqrt{n}}$  
for all $p\in \{1\}\cup [2,\infty)$ and a constant $\KRsig$ depending only on $(R,\sig)$ defined in \cref{bb}. 
\end{theorem}
\begin{rem}[Auxiliary tail bounds]\label{aux_tail_bounds}
We include the auxiliary bounds  $\onetail(u)$ and $\twotail(u)$ to emphasize that our efficient  bounds can be paired with any valid tail bounds to simultaneously reap the large-sample benefits of the former and the small-sample benefits of the latter.
A convenient default choice is to set  $\twotail(u) = 2 \onetail(u)$ and $\onetail(u)$ equal to the minimum of the zero-bias (\cref{efficient-zero-tail}), Hoeffding \cref{eq:bern_tails}, Bernstein \cref{eq:bern_tails}, 
 Berry-Esseen \cref{eq:berryesseen},  non-uniform Berry-Esseen \cref{eq:nonuniformberryesseen}, and Bentkus \cref{eq:bentkus_tails} bounds.  
\end{rem}
\cref{cpt_concentration} immediately gives rise to the following looser but evidently efficient tail bounds.
\begin{corollary}[Efficiency and relative error]\label{wass-efficiency-example}
Under the conditions of \cref{cpt_concentration}, 
\begin{talign}\notag
\P(S_n > \sig u) 
    &\leq 
\Phi^c(u - \delta_{u,n}) 
    +
\frac{\varphi(u)}{\sqrt{n}} 
    \qtext{for}
\delta_{u,n}\defeq\frac{e \KRsig}{\sqrt{n}} \lceil\log\big(\frac{\sqrt{n}}{\varphi(u)}\big)\rceil
    \\&\leq 
\Phi^c(u)\cdot\big(e^{ {(u+1)\delta_{u,n}}} 
    +
\frac{1}{\sqrt{n}}\frac{\varphi(u)}{\Phi^c(u)}\big)
    \qtext{for all}
u > \delta_{u,n}. \label{eq:rel-error-final}
\end{talign}
\end{corollary}
\begin{proof}
Fix any $u > \delta_{u,n}$. 
Since $\wassbd \leq \KRsig\frac{p}{\sqrt{n}}$ for all $p\in\naturals$ by \cref{cpt_concentration}, we may choose $p = \lceil\log(\frac{\sqrt{n}}{\varphi(u)})\rceil \in \naturals$ and $\rho = 1 - e \,\wassbd / u \in (0,1)$ in the upper bound \cref{eq:wass-tail-bound} to deduce the first inequality. 
The log-concavity of $\Phi^c$ \citep[Thm.~2]{bagnoli2005log} additionally implies that
\begin{align}
    \Phi^c(u-\delta_{u,n})\leq\Phi^c(u)\exp(\delta_{u,n}\,\varphi(u)/\Phi^c(u)).
\end{align}
The second inequality therefore follows from the Mills ratio bound $\frac{\varphi(u)}{\Phi^c(u)}\le u+1$  \cite{10.1214/aoms/1177731611}.
\end{proof}

While the $n$ dependence of \cref{wass-efficiency-example} is no better than that established for \cref{efficient-zero-tail}, the $u$ dependence is improved, yielding tighter bounds for larger deviations $u$ 
and a relative error bound \cref{eq:rel-error-final} in the spirit of classical \Cramer-type inequalities. 
For practical applications, the Wasserstein tail bound of \cref{cpt_concentration} is substantially tighter than \cref{wass-efficiency-example}, applies to the full range of $u \geq 0$, and still admits a straightforward implementation.\cref{github} 
Hence, for practical use, we recommend pairing the zero-bias and Wasserstein inequalities as in \cref{aux_tail_bounds} to simultaneously inherit the benefits of each efficient bound.

In addition, to obtain efficient, computable quantile bounds, we recommend simply inverting the tail bounds of \cref{cpt_concentration,aux_tail_bounds}.
Our next result, also proved in \cref{last_night}, makes this precise.

\begin{theorem}[Efficient Wasserstein quantile bound]\label{cpt-quantile} 
Under \cref{assump:bounded-dev}, for any $\delta\in(0,1)$ and auxiliary deterministic bounds
$\oneq$ and $\twoq$ satisfying $P\big(S_n> \oneq(R,\delta,\sigma)\big) \leq \delta$ and $P\big(|S_n|> \twoq(R,\delta,\sigma)\big) \leq \delta$, 
we have
\begin{talign}
\P(S_n> q_n(R,\delta,\sigma))\le \delta
\qtext{and} 
\P(|S_n|> q^d_n(R,\delta,\sigma))\le \delta,\end{talign}
for $\wassbd$ as in \cref{cpt_concentration} 
and
\begin{talign}\label{eq:tighter_efficient_quantile}&
q_n(R,\delta,\sigma)\defeq \min\Big(\inf_{\substack{p\in \naturals,\, \rho\in (0,1)}}\frac{\sig\wassbd(1-\indic{p=2}\delta(1-\rho))^{1/p}}{(\delta(1-\rho))^{1/p}}+\sigma \Phi^{-1}(1-\delta \rho), ~\oneq(R,\delta,\sigma) \Big),
\\&q^d_n(R,\delta,\sigma)\defeq \min\Big(\inf_{\substack{p\in \naturals,\, \rho\in (0,1)}}\frac{\sigma \wassbd}{(\delta(1-\rho))^{1/p}}+\sigma \Phi^{-1}(1-\frac{\delta \rho}{2}), ~\twoq(R,\delta,\sigma) \Big).
\end{talign}
\end{theorem}
\begin{rem}[Auxiliary quantile bounds]\label{aux_q_bounds}
Our efficient  bounds can be paired with any deterministic quantile bounds $\oneq(R,\delta,\sigma)$ and $\twoq(R,\delta,\sigma)$ to inherit the benefits of each.
Convenient defaults are provided by 
numerically inverting the default auxiliary tail bounds $\onetail(u)$ and $\twotail(u)$ of
\cref{aux_tail_bounds}.
\end{rem}

In the sequel, we build upon the quantile bounds of \cref{cpt-quantile,aux_q_bounds} to develop efficient \emph{empirical Berry-Esseen bounds} that require no prior knowledge of the variance parameter $\sigma$.
\subsection{Proof of \cref{cpt_concentration,cpt-quantile}} \label{last_night}
\cref{cpt_concentration,cpt-quantile} will follow directly from three auxiliary results.
The first, proved in \cref{proof-snow_day}, shows that any upper bound on the Wasserstein distance $\frac{1}{\sig}\wass (S_n,\Gsn(0,\sigma^2))$ also yields tail and quantile bounds for $S_n$.

\begin{lemma}[Tail and quantile bounds from Wasserstein bounds]\label{snow_day}
Under \cref{assump:bounded-dev}, 
suppose that
$\frac{1}{\sig}\wass (S_n,\Gsn(0,\sigma^2))\le \omega_p$
for some $p\geq 1$ and $\omega_p \geq 0$.
Then, for any $u\geq 0$,
\begin{talign}
    \P\big(S_n> \sig u\big)&\le\inf_{\substack{\rho\in (0,1)}}\Big\{\Phi^c\big(\rho u\big)+\frac{\omega_p^p}{ (1-\rho)^pu^p+\omega_p^p\,\indic{p=2}}\Big\}
    \qtext{and}\\ 
    \P\big(|S_n|>\sig u\big)&\le\inf_{\substack{\rho\in (0,1)}}\Big\{2\Phi^c\big(\rho u\big)+\frac{\omega_p^p}{ (1-\rho)^pu^p}\Big\}.
\end{talign}
Moreover, for all $\delta\in(0,1)$, 
\begin{talign}
\P(S_n&> 
\inf_{\rho\in(0,1)} 
\frac{\sig\omega_p(1-\indic{p=2}\delta(1-\rho))^{1/p}}{(\delta(1-\rho))^{1/p}}+
\sigma \Phi^{-1}(1-\delta \rho)
)
\le \delta 
\qtext{and} \\
\P(|S_n|&> 
\inf_{\rho\in(0,1)} 
\frac{\sig\omega_p}{(\delta(1-\rho))^{1/p}}+
\sigma \Phi^{-1}(1-\half\delta \rho)
)
\le \delta.
\end{talign} 
\end{lemma}

Our second result, proved in \cref{proof-lemm:dino_cpt}, establishes that $\wassbd$ is indeed an upper bound on $\frac{1}{\sig}\wass(S_n,\Gsn(0,\sigma^2))$. 
\begin{lemma}[Wasserstein upper bound]\label{lemm:dino_cpt} 
Under \cref{assump:bounded-dev}, for $p=1$ or $p\ge2$, 
\begin{talign}\label{sponsor}
\frac{1}{\sig}
&\wass(S_n,\Gsn(0,\sigma^2))
\le 
\wassbd
\qtext{and}
\omega_p^{R,\kappa}(\sig)\leq \omega_p^{R,\kappa,2}(\sig)
\qtext{for all}
\kappa > \frac{\tilde R^2}{n}
\end{talign}
where
\begin{align+}\label{cabris}
\tag{$\omega_p^R$}
\omega_p^{R}(\sigma)
    \defeq 
\begin{cases}
\frac{\Rsig}{\sig\sqrt{n}}&\textup{if}\quad p=1,\\
\inf_{\kappa > \frac{\tilde R^2}{n}} \  \omega^{R,\kappa,1}_p(\sigma)& \textup{if}\quad n+1\ge p\ge 2,\\ 
\sqrt{p-1}(1+\frac{\Rsig}{\sigma})& \textup{if}\quad p>n+1,
\end{cases}
\end{align+}
and, for each $p \in [2,n+1]$, $\kappa > 0$, and $j\in\{1,2\}$,  
\begin{talign}
\omega^{R,\kappa,j}_p(\sigma)
&\defeq \Big[\|Z\|_p
\Big(\frac{\pi}{2}-\sin^{-1}\Big(\sqrt{1-\frac{\tilde{R}^2}{n\kappa}}\Big)\Big)
    +
(b_{2,j}^{\kappa,p,\tilde R})+(b_{3,j}^{\kappa,p,\tilde R})\Big]
\begin{cases}\frac{1}{\sqrt{1-\frac{\tilde R^2}{n\kappa}}}&\textup{if }\quad p>2
    \\ 1 &\textup{if }\quad p=2,
\end{cases} 
\end{talign}
for 
$\tilde R\defeq R/\sigma$, 
$Z\sim\Gsn(0,1)$, 
and  
$(b_{2,1}^{\kappa,p,\tilde R})$, $(b_{3,1}^{\kappa,p,\tilde R})$,
$(b_{2,2}^{\kappa,p,\tilde R})$, and $(b_{3,2}^{\kappa,p,\tilde R})$ 
defined in \cref{fr_al_2,fr_al_3,fr_al_2b,fr_al_3b} respectively.
\end{lemma}

Our final lemma, proved in \cref{proof-bb}, verifies the growth rate of $\wassbd$. %
\begin{lemma}[Growth of $\wassbd$]\label{bb}
Instantiate the notation of \cref{lemm:dino_cpt} and define  $\widetilde \Rsig\defeq \frac{\Rsig}{\sig}$. 
If $p=1$ or $ p\ge 2$, then 
$\wassbd\le \frac{\KRsig p}{\sqrt{n}}$
    for 
\begin{align+}\label{KRsig}\notag
K_{R,\sigma}
    \defeq
1
    &\textstyle+
\max\Big(
    \tRsig,
\frac{3e}{4}\sqrt{\Rsigsqd-\sig^2}
    +
\sqrt{2\max(\Rsigsqd-\sig^2,\sig^2)}
    \\ 
    \tag{$K_{R,\sig}$}&
\textstyle
    +
(\frac{3}{16}
    +
\frac{\sqrt{2}\log(4e)\tilde R^{1-2/p}(8\pi)^{1/4}}{\sqrt{3e}}
    +
\frac{\tilde Re}{2\sqrt{2}}
    +
\frac{\tilde R^{2-2/p}}{\sqrt{2}}) %
\frac{\sig\pi^{1/4}e^{19/300}(e^{\frac{\tilde R^2}{2}}-1)}{\sqrt{3}}
    \Big).
 \end{align+} 
 \end{lemma}

\subsection{\pcref{snow_day}}\label{proof-snow_day}
Fix any $u \geq 0$, $\rho\in (0,1)$, $p\geq 1$, and $\eps > 0$. 
By the definition of the Wasserstein distance, we can find $G\sim \Gsn(0,\sigma^2)$ that respects $\|S_n-G\|_p\le \epsilon+\wass (S_n,\Gsn(0,\sigma^2))$. 
Hence, combining the union bound with Markov's inequality for $p\neq 2$ and Cantelli's inequality~\citep[(18)]{cantelli1929sui} for $p=2$, we find that
\begin{talign}
&\Parg{S_n> \sig u }
    = \Parg{G + S_n-G >\sig u } 
    \leq \Parg{G > \rho \sig u  } + \Parg{S_n-G > (1-\rho) \sig u } \\
    &\leq  
    \Phi^c(\rho u) + \frac{\E(|S_n - G|^p)}{(1-\rho)^p\sig^p u^p+\E(|S_n - G|^p)\,\indic{p=2}} 
    \leq  \Phi^c(\rho u) + \frac{(\epsilon+\wass(S_n,\Gsn(0,\sigma^2)))^p}{(1-\rho)^p\sig^p u^p+(\epsilon+\wass(S_n,\Gsn(0,\sigma^2)))^p\,\indic{p=2}}.
\end{talign} 
As this holds for arbitrary $\epsilon>0$ we obtain 
\begin{talign}
    \P(S_n> \sig u)\le \Phi^c(\rho u) +\frac{\wass (S_n,\Gsn(0,\sigma^2))^p}{ (1-\rho)^p\sig^p u^p+\wass(S_n,\Gsn(0,\sigma^2))^p\,\indic{p=2}}.
\end{talign} 
The triangle inequality and parallel reasoning yield 
\begin{talign}
\Parg{|S_n|> \sig u }
    &\leq \Parg{|G| > \rho \sig u  } + \Parg{|S_n-G |> (1-\rho) \sig u } \leq  2\Phi^c(\frac{\rho \sig u}{\sigma}) + \frac{(\epsilon+\wass(S_n,\Gsn(0,\sigma^2)))^p}{(1-\rho)^p\sig^p u^p}.
\end{talign}
Since $\eps > 0$ is arbitrary, we conclude that
\begin{talign}
    \P(|S_n|> \sigma u)\le 2\Phi^c(\rho u) +\frac{\wass (S_n,\Gsn(0,\sigma^2))^p}{ (1-\rho)^p
    \sigma^pu^p}.
\end{talign} 

Now, fix any $\delta\in(0,1)$ and write $t_{\delta,\rho}\defeq (\frac{\wass (S_n,\Gsn(0,\sigma^2))^p}{\delta(1-\rho)}-\wass (S_n,\Gsn(0,\sigma^2))^p\indic{p=2})^{1/p}+\sigma\Phi^{-1}(1-\delta \rho) $. 
We apply the union bound, Markov's inequality for $p\ne 2$, Cantelli's inequality for $p=2$, and the Wasserstein condition in turn to find
\begin{talign}
    &\P(S_n>t_{\delta,\rho})= \P( S_n-G+G> t_{\delta,\rho})
    \\&\le \P(G> \sigma \Phi^{-1}(1-\delta \rho))+\P\big(S_n-G>(\frac{\wass (S_n,\Gsn(0,\sigma^2))^p}{\delta(1-\rho)}-\wass (S_n,\Gsn(0,\sigma^2))^p\indic{p=2})^{1/p}\big)
     \\&\le  \delta\rho +\frac{\delta(1-\rho)(\wass (S_n,\mathcal{N}(0,\sigma^2))+\epsilon)^p}{ \wass (S_n,\Gsn(0,\sigma^2))^p}.
\end{talign} %
Since this holds for any $\epsilon>0$, we deduce that 
\begin{talign}
\P(S_n>t_{\delta,\rho})\le  \delta\rho+\delta(1-\rho)=\delta. 
\end{talign} 
Moreover, since $\rho$ was arbitrary and each $t_{\delta,\rho}$ is deterministic, we further have
\begin{talign}
\P\big(S_n>\inf_{\rho\in(0,1)}t_{\delta,\rho}\big)\le \delta. 
\end{talign} 
Finally, the triangle inequality and parallel reasoning yield
\begin{talign}
\P\big(|S_n|>\inf_{\rho\in(0,1)}t_{\delta,\rho}^d\big)\le \delta
\qtext{for}
t_{\delta,\rho}^d
\defeq
\frac{\wass (S_n,\Gsn(0,\sigma^2))}{(\delta(1-\rho))^{1/p}}+\sigma\Phi^{-1}(1-\half\delta \rho).
\end{talign} 

\subsection{\pcref{bb}}\label{proof-bb}
If $p=1$, then $\wassbd = \frac{\widetilde\Rsig}{\sqrt{n}}$. If $p>n+1$, 
\begin{talign}
\wassbd = \sqrt{p-1}\,(1+\widetilde\Rsig)\leq \frac{p}{\sqrt{n}}(1+\widetilde\Rsig).
\end{talign}
Now suppose $p\in[2,n+1]$.  
By \cref{lemm:dino_cpt}, we have, for all $\kappa> \frac{R^2}{\sigma^2n}$,
\begin{talign}&
 \wassbd\frac{(M_{n,\kappa}\indic{p >2}+\indic{p=2})\sqrt{n}}{p} \\ 
    &\le 
 \frac{\sqrt{(p-1)(p+2)}\sigma}{2p}\Big[
 \frac{\sqrt{p+2}n^{1/p}}{2\sqrt{n}}(
 \max(\widetilde \Rsigsqd-1,1))^{1-1/p}+\sqrt{\widetilde \Rsigsqd-1}\frac{\sqrt{e}(2e)^{1/p}}{\sqrt{2}}\Big]\\
    &+
\frac{C_{n,p}\sigma}{p}\frac{\tilde R^2}{3\sqrt{n}\kappa}e^{19/300}\pi^{1/4}[e^{\frac{1}{2}(p-1)(1-\frac{\tilde R^2}{n\kappa})\kappa}-1]
\log(\frac{1+\mnkappa}{1-\mnkappa})
\\
    &+
\frac{ R\big(1+\frac{\tilde R \tilde U_{n,p}}{\sqrt{n}}\big)}{\sqrt{\kappa}}\frac{\pi^{1/4}e^{19/300}\sqrt{p-1}}{4\sqrt{3}}\big[e^{\frac{1}{2}(p-1)\kappa \mnkappa^2}-1\big]+\frac{\sqrt{n}\sqrt{p-1}}{p}\Big[\frac{\pi}{2}-\sin^{-1}(\mnkappa)\Big]
\end{talign} 
where $\mnkappa=\sqrt{1-\frac{\tilde R^2}{n\kappa}}$ 
and 
$C_{n,p}$ is a constant defined in \cref{eq:Cnp} that satisfies $C_{n,p} \leq \tilde R^{1-2/p}\sqrt{8\pi}^{1/p}\frac{\sqrt{p}}{\sqrt{e}}$ by \citet[Thm.~1.5]{batir2008inequalities}. 
Since we additionally have $\frac{\pi}{2}-\sin^{-1}(\mnkappa)\le \frac{\tilde R^2}{n\kappa\mnkappa}$ by the mean value theorem, we can write 
\begin{talign}
    &\omega_p^R(\sigma)\frac{\sqrt{n}\sqrt{1-\frac{\tilde R^2(p-1)}{nC}\indic{p>2}}}{p} \\& \le   \frac{\sqrt{(p-1)(p+2)}\sigma}{2p}\Big(\sqrt{\widetilde \Rsigsqd-1}\frac{\sqrt{e}(2e)^{1/p}}{\sqrt{2}}+\frac{\tilde R}{\sqrt{C}}\frac{\pi^{1/4}e^{19/300}(p-1)}{4p\sqrt{3}}\big[e^{\frac{C}{2} }-1\big]\Big)\\&+\frac{\sqrt{p}\log(\frac{4nC}{(p-1)\tilde R^2})\sigma}{\sqrt{n}}\Big\{\frac{\tilde R^{3-2/p}\sqrt{8\pi}^{1/p}}{3\sqrt{e}C}\frac{p-1}{p}e^{19/300}\pi^{1/4}[e^{\frac{C}{2}}-1]\Big\}
\\&+\frac{\sqrt{p+2}\sigma}{\sqrt{n}^{1-2/p}}\frac{(\max(\widetilde \Rsigsqd-1,1))^{1-1/p}}{2}
+\frac{\sqrt{p+2}\sigma}{\sqrt{n}}\Big[\frac{\tilde R^2 \frac{\sqrt{e}(2e)^{1/p}}{\sqrt{2}}}{\sqrt{C}}\frac{\pi^{1/4}e^{19/300}(p-1)}{4p\sqrt{3}}\big[e^{\frac{C}{2}}-1\big]\Big]
 \\&+\frac{(p+2)\sigma\tilde R^{3-2/p}}{n^{1-1/p}}\frac{1}{2\sqrt{C}}\frac{\pi^{1/4}e^{19/300}(p-1)}{4p\sqrt{3}}\big[e^{\frac{C}{2}}-1\big]+\frac{\tilde R^2\sqrt{p-1}^3}{p\sqrt{n}C\sqrt{1-\frac{\tilde R^2(p-1)}{nC}}}
\end{talign}
in terms of the reparameterization $C\defeq\kappa(p-1)$. 
The advertised result now follows from the choosing $C= \tilde R^2$, since  
$\kappa=
\frac{\tilde R^2}{p-1}
\ge
\frac{\tilde R^2}{n}$ 
and 
 $(p+2)(\log(p)-1)\le n^2/2$ for $p\leq n+1$   and 
 $p\rightarrow \sqrt{p+2}n^{1/p}$ is decreasing for $(p+2)(\log(p)-1)\le n^2/2$. 

%% file: arxiv-sections/empirical.tex
\section{Application: Efficient \emph{Empirical} Berry-Esseen Bounds}\label{empi}
The preceding sections assumed that the variance parameter $\sigma^2$ was a known quantity, but, in many applications, the variance is unknown and can only be estimated from data.  
In this section, we leverage our efficient known-variance bounds to develop efficient quantile bounds that are valid even when $\sigma$ is unknown. 
We refer to these constructions as   \emph{empirical Berry-Esseen}  bounds as they combine Gaussian approximation in the spirit of Berry-Esseen \cref{eq:berryesseen} with empirical variance estimation.

We begin by showing how to convert generic known-variance quantile bounds into valid empirical-variance quantile bounds.
Our construction in \cref{unknown} makes use of a generic confidence interval that contains the unknown variance with high probability. 
\begin{lemma}[Empirical quantile bounds]\label{unknown}
Consider any nonnegative interval $[\siglow^2, \sigup^2]$ and  any nonnegative quantile bounds $\oneq$ and $\twoq$  satisfying
\begin{talign}
\P\big(\sig\notin[\siglow,\sigup]\big) \leq a,
    \quad
\P\big(S_n> \oneq(R,\delta,\sigma)\big)\leq \delta, 
    \qtext{and}
\P\big(|S_n|> \twoq(R,\delta,\sigma)\big)\leq \delta 
\end{talign}
whenever $S_n$ satisfies \cref{assump:bounded-dev} and $\delta, a\in(0,1)$.
Then, under \cref{assump:bounded-dev}, we have, for each confidence level $\delta\in(0,1)$ and $a_n \in (0,\delta)$, 
\begin{talign}
\P(S_n > \oneqhat(R,\delta,a_n))\le \delta   
&\qtext{for}
\oneqhat(R,\delta,a) 
    \defeq
\sup_{\tilde \sigma\in [\siglow,\sigup]} \oneq(R,\delta-a,\tilde\sigma) \qtext{and}\\
\P(|S_n| > \twoqhat(R,\delta,a_n))\le \delta   
&\qtext{for}
\twoqhat(R,\delta,a) 
    \defeq
\sup_{\tilde \sigma\in [\siglow,\sigup]} \twoq(R,\delta-a,\tilde\sigma).
\end{talign}
\end{lemma}
\begin{proof}
 Fix any $\delta>0$ and $a \in(0,\delta)$. Under \cref{assump:bounded-dev}, the union bound implies
\begin{talign}          
\P(|S_n| > \twoqhat(R,\delta,a))
&\le \P(|S_n| > \twoq(R,\delta-a,\sigma))+\P(\sig\notin[\siglow,\sigup])
\le 
\delta.
\end{talign}
The one-sided result is obtained identically using the event  
$\{S_n > \oneq(R,\delta-a,\sigma)\}$.
\end{proof}

Finally, we show how to convert efficient known-variance bounds into efficient empirical Berry-Esseen bounds through appropriate choice of the variance confidence interval and the confidence parameter $a_n$. 
For this purpose we use the sharp (and valid) empirical Bernstein variance confidence intervals of \citet[Sec.~4.4]{martinez2025sharp} which depend on 
the auxiliary functions 
\begin{talign}
\Psi_E(x)\defeq-\log(1-x)-x 
\qtext{and}
\Psi_P(x)\defeq e^x - x - 1
\end{talign}
and the following readily-computed quantities for $i\in\{1,\dots,n\}$ and $a\in(0,1)$:
\newcommand{\llam}{\lam'} %
\newcommand{\tlam}{\tilde{\lam}} %
\newcommand{\lamhalf}{\lambda_{i,\frac{a}{2}}}
\begin{itemize}[leftmargin=12.5pt]
    \item the $\mu\defeq\E[W_1]$ estimates: 
    \ $\muhat_i \defeq \frac{R}{2i} + \frac{1}{i}\sum_{j=1}^{i-1}W_j$,
    \item the $\sig^2$  estimates: 
    \ $\hat\sigma^2_i \defeq \frac{R^2}{4i} + \frac{1}{i}\sum_{j=1}^{i-1}(W_j-\mhat_j)^2$,
    \item the $m_4^2\defeq\Var((W_1-\mu)^2)$ estimates: 
    \ $\mhat_{4,i}^2\defeq  \frac{R^4}{8i} + \frac{1}{i}\sum_{j=1}^{i-1}((W_j-\muhat_j)^2-\hat\sigma_j^2)^2$,
    \item the upper estimate weights:
    \ $\lambda_{i,a}
        \defeq
    \sqrt{\frac{2R^4\log(2/a)}{{n\mhat_{4,i}^2}}}\wedge \frac{1}{2}$,
    \item the auxiliary weights: 
    \ $\tlam_{i,a} 
        \defeq
    \sqrt{\frac{2R^2\log((2+2\log(n))/a)}{\sighat_i^2 i \log(1+i)}} \wedge 2$,
    \item the lower estimate weights: 
    \ $\llam_{i,a}
        \defeq
    \begin{cases}
        \lambda_{i,\frac{a\log(n)}{1+\log(n)}}%
            &\textup{if } 
        \frac{\log(\frac{2+2\log(n)}{a})+\sum_{j=1}^{i-1}\Psi_P(\tlam_{j,a}){\hat\sigma^2_i}{/R^2}}{\sum_{j=1}^{i-1}\tlam_{j,a}}\le 1\\ 
        0 
            &\textup{otherwise},
    \end{cases}$
    \item the $\sigma^2$ upper estimate:
\end{itemize}
    \begin{align}\label{eq:sigup}
    \sigup^2
        \defeq 
    \frac{
    R^2\log(2/a)
        +
    \sum_{i=1}^n \lam_{i,a} (W_i-\muhat_i)^2
        +    
    \Psi_E(\lam_{i,a})((W_i-\muhat_i)^2-\hat\sigma_i^2)^2/R^2}
    {\sum_{i=1}^n\lam_{i,a}},
    \end{align}
\begin{itemize}
    \item and the $\sigma^2$ lower estimate: 
\end{itemize}  
    \begin{align}\label{eq:siglow}
    \vspace{-\baselineskip}\siglow^2
        \defeq 
    \frac{
        -
    R^2\log(2/a)
        +
    \sum_{i=1}^n\! \llam_{i,a} (W_i-\muhat_i)^2
        -    
    \!\Psi_E(\llam_{i,a})((W_i-\muhat_i)^2-\hat\sigma_i^2)^2/R^2}
    {\sum_{i=1}^n\llam_{i,a}}.
    \end{align}

\begin{theorem}[Efficient empirical Berry-Esseen bounds]\label{efficient-unknown}
Instantiate \cref{assump:bounded-dev}, and 
consider any bounds $(\oneq,\twoq)$ satisfying  
$\oneq(R,\delta_n,\sigma_n)\to \sigma\Phi^{-1}(1-\delta)$
and 
$\twoq(R,\delta_n,\sigma_n)\to \sigma\Phi^{-1}(1-\frac{\delta}{2})$
whenever $(\sigma_n, \delta_n)\to (\sigma, \delta)$. If 
$a_n \to 0$ and 
$\frac{\log(a_n)}{n} \to 0$, 
then 
\begin{talign}
\oneqhat(R,\delta,a_n) 
    &\defeq
\sup_{\tilde \sigma\in [\siglown,\sigupn]} \oneq(R,\delta-a_n,\tilde\sigma) 
   \ \toas \ 
\sigma\Phi^{-1}(1-\delta)
    \qtext{and}\\
\twoqhat(R,\delta,a_n) 
    &\defeq
\sup_{\tilde \sigma\in [\siglown,\sigupn]} \twoq(R,\delta-a_n,\tilde\sigma)
    \ \toas \ 
\sigma\Phi^{-1}(1-\frac{\delta}{2})
\end{talign}
when $\sigupn$ and $\siglown$ are chosen as in \cref{eq:sigup,eq:siglow}.
\end{theorem}
\begin{proof}
The strong law of large numbers \citep[Thm.~2.4.1]{durrett2019probability}  
and Prop.~C.1 of \citep{martinez2025sharp} imply that 
$\muhat_n \toas\mu$, 
$\hat\sigma_n^2\toas \sigma^2$, 
and 
$\mhat_{4,n}^2\toas m_4^2$. 
Thus, with probability $1$, there exists a finite integer $\Nmint$ 
for which $\sighat_i^2 > \sighat^2/2$ and $\tlam_{i, a_n} < 2$ whenever $\tilde{N}_n\defeq\Nmint\vee \log(\frac{2+2\log(n)}{a_n}) \leq i \leq n$.
Defining 
$\td_i \defeq \sqrt{\frac{2R^2i\log((2+2\log(n))/a_n)}{\log(i+1)}}$ for each $i\le n$, we can therefore write
\begin{talign}
\frac{1}{\td_i}&\sum_{j=1}^{i-1}\tlam_{j,a_n}^2
    = 
\frac{1}{\td_i}\sum_{j=1}^{(i-1)\wedge\tilde{N}_n}\tlam_{j,a_n}^2
    + 
\sqrt{\frac{\log(i+1)}{i}}\sum_{j=\tilde{N}_n+1}^{i-1}\frac{\sqrt{2R^2\log((2+2\log(n))/a_n)}}{\sighat_j^2 j\log(j+1)}
\\
    &\leq
\frac{4\tilde{N}_n\sqrt{\log(i+1)}}{\sqrt{2R^2i\log((2+2\log(n))/a_n)}}
    +
\frac{2(1+ \log(i-1))}{\sig^2}\sqrt{\frac{2R^2\log((2+2\log(n))/a_n)\log(i+1)}{i\log(2) }}
    \qtext{and}\\
\frac{1}{\td_i}&\sum_{j=1}^{i-1}\tlam_{j,a_n}
    =
\frac{1}{\td_i}\sum_{j=1}^{(i-1)\wedge\tilde{N}_n}\tlam_{j,a_n}
    +
\sqrt{\frac{\log(i+1)}{i}}\sum_{j=\tilde{N}_n+1}^{i-1}\frac{1}{\sighat_j \sqrt{j\log(j+1)}}
    \geq
\frac{2\sqrt{i}-2\sqrt{\tilde{N}_n+1}}{R\sqrt{i}}.
\end{talign}
Since $\Psi_P(x)\leq x^2$ for all $x\in[0,2]$, we have 
$\sum_{j=1}^{i-1}\Psi_P(\tlam_{j,a_n})
    \leas 
\sum_{j=1}^{i-1}\tlam_{j,a_n}^2$ for each $i$ and therefore 
$\llam_{i,a_n} = \lambda_{i,a_n'} < \half$ 
whenever $i \geq \Nmin' \defeq c_{R,\sig}\tilde{N}_n\log(n+1)^3$ for some constant $\tilde{c}_{R,\sig}$ depending only on $(R,\sig)$ and $a_n' \defeq \frac{a_n\log(n)}{1+\log(n)}$. 

Similarly, \cref{assump:bounded-dev} and the inequalities $0\leq\Psi_E(x)\le x^2$ 
for all $x\in[0,\half]$ imply
\begin{talign}
0
    \leq
\sum_{i=1}^n\Psi_E(\lam_{i,a_n})\frac{((W_i-\muhat_i)^2-\hat\sigma_i^2)^2}{R^2}
    &\leas
\sum_{i=1}^n\min(
    R^4\lam_{i,a_n}^2, 
    \lam_{i,a_n} \frac{((W_i-\muhat_i)^2-\hat\sigma_i^2)^2}{2R^2}
    ).\label{eq:psiEbound}
\end{talign}
We will now show that $\sigupn\toas\sigma$ and $\siglown\toas\sigma$ by considering two cases.

\textbf{Case 1: $m_4 > 0$.}\ 
Since $\frac{\log(a_n)}{n}\to 0$, 
with probability $1$ a finite integer $\Nmin$ exists 
for which $\lam_{i,a_n} < \half$ whenever $\Nmin < i \leq n$.
Therefore, since 
$d_n \defeq \sqrt{2R^4 n\log(2/a_n)}\to\infty$,
$\log(a_n)/n\to 0$, and 
each $\lam_{i,a_n} \leq \half$, 
we have
\begin{talign}
\frac{1}{d_n}\sum_{i=1}^{n}\lam_{i,a_n}
    &=
\frac{1}{d_n}\sum_{i=1}^{n\wedge\Nmin}\lam_{i,a_n}
    +
\frac{1}{{n}}\sum_{i=\Nmin+1}^n\frac{1}{\mhat_{4,i}}
    \toas
\frac{1}{m_4}
    \qtext{and} \\
\frac{1}{d_n}\sum_{i=1}^n\lam_{i,a_n}^2
    &= 
\frac{1}{d_n}\sum_{i=1}^{n\wedge\Nmin}\lam_{i,a_n}^2
    + 
\sqrt{\frac{2R^4\log(2/a_n)}{n}}\frac{1}{n}\sum_{i=\Nmin+1}^n\frac{1}{\mhat_{4,i}^2}
    \toas 
0.
\end{talign}
The almost sure boundedness, $(W_i-\muhat_i)^2\in [0,R]$ for all $i$, further implies that 
\begin{talign}
&\frac{1}{d_n}\sum_{i=1}^n \lam_{i,a_n}(W_i-\muhat_i)^2 \\
    &=
\frac{1}{d_n}\sum_{i=1}^{n\wedge\Nmin}\lam_{i,a_n} (W_i-\muhat_i)^2
    +
\frac{1}{{n}}\sum_{i=\Nmin+1}^n[(\frac{1}{\mhat_{4,i}}-\frac{1}{m_4})(W_i-\muhat_i)^2
    +
\frac{(W_i-\muhat_i)^2}{m_4}]
    \toas
\frac{\sig^2}{m_4}.
\end{talign} 
These results, together with the $\Psi_E$ estimates~\cref{eq:psiEbound} and the assumption $\frac{\log(a_n)}{n}\to 0$, ensure 
\begin{talign}
|\sigupn^2-\sig^2|
    \leas
|\frac{\frac{1}{d_n}\sum_{i=1}^n \lam_{i,a_n}(W_i-\muhat_i)^2}{\frac{1}{d_n}\sum_{i=1}^n\lam_{i,a_n}}
    -
\sig^2|
    +
\frac{\frac{1}{d_n}R^2\log(2/a_n)+R^4\frac{1}{d_n}\sum_{i=1}^n\lam_{i,a_n}^2}
    {\frac{1}{d_n}\sum_{i=1}^n\lam_{i,a_n}}
 \toas 0.  
\end{talign}
Identical reasoning with $(\Nmin',a_n')$ substituted for $(\Nmin, a_n)$ ensures that  $\siglown\toas \sig$.

\textbf{Case 2: $m_4 = 0$.} \ 
Since $(W_1-\mu)^2 \eqas 0$, we have 
$\hat\sigma^2_n - \sig^2 \eqas \frac{1}{n}(\frac{R^2}{4}-\sig^2)\to 0$,
\begin{talign}
|(W_n-\muhat_n)^2 - \sig^2|
    \eqas  
|(\mu-\muhat_n)^2 + 2(W_n-\mu)(\mu-\muhat_n)|         \leas 
(\mu-\muhat_n)^2 + 2R|\mu-\muhat_j|
    \toas
0,
\end{talign}
and therefore $((W_n-\muhat_n)^2-\sighat_n^2)^2 \toas 0$. 
Moreover, since $\mhat_{4,i}^2\leas R^4$ for each $i$, we have 
\begin{talign}
\sum_{i=1}^n\lam_{i,a_n}
    \geas
\sum_{i=1}^n
    (\sqrt{{16\log(2/a_n)/}{(R^4n)}} \,\wedge \half)
    = 
\sqrt{16n\log(2/a_n)/R^4}\wedge \frac{n}{2}
    \to 
\infty.
\end{talign} 
Therefore, our assumption $\frac{\log(a_n)}{n}\to 0$, the $\Psi_E$ estimates \cref{eq:psiEbound}, and the Silverman-Toeplitz theorem \citep[Thm.~2]{hardy1956divergent} imply that
\begin{talign}
|\sigupn^2-\sig^2|
    \leq
|\frac{\sum_{i=1}^n \lam_{i,a_n}(W_i-\muhat_i)^2}{\sum_{i=1}^n\lam_{i,a_n}}
    -\sig^2|
    +
\frac{R^2\log(2/a_n)+\sum_{i=1}^n\lam_{i,a_n}{((W_i-\muhat_i)^2-\hat\sigma_i^2)^2}{/(2R^2)}}
    {\sum_{i=1}^n\lam_{i,a_n}}
 \toas 0.  
\end{talign}
Parallel reasoning implies that $\siglown^2\toas\sig^2$ as
\begin{talign}
\sum_{i=1}^n\lam_{i,a_n}'
    \ge
\sum_{i=\Nmin'}^n
    (\sqrt{\frac{16\log(2/a_n')}{R^4n}} \,\wedge \half)
    = 
\sqrt{\frac{16(n-\Nmin')_+^2\log(2/a_n')}{R^4n}}\wedge \frac{n-\Nmin'}{2}
    \to 
\infty.
\end{talign} 

Since $\sigupn$ and $\siglown$ both converge almost surely to $\sigma$, any
\begin{align}
\sigmidone \in \argmax_{\tilde\sigma\in[\siglown,\sigupn]} \oneq(R,\delta-a_n,\tilde\sigma)
\qtext{and}
\sigmid \in \argmax_{\tilde\sigma\in[\siglown,\sigupn]} \twoq(R,\delta-a_n,\tilde\sigma).
\end{align}
must converge almost surely to $\sigma$ as well. 
Finally, since $a_n\to 0$, 
\begin{talign}
\empq &= \twoq(R,\delta-a_n,\sigmid) \toas -\sigma\Phi^{-1}(\frac{\delta}{2})    
    \qtext{and} 
\oneqhat(R,\delta,a_n) 
\toas -\sigma\Phi^{-1}(\delta).
\end{talign}

\end{proof}
\begin{rem}[Default settings]\label{default-efficient-unknown}
As convenient default settings, one can choose $a_n = \frac{\delta}{\sqrt{n}}\Phi^{-1}(1-\delta)$ for a one-sided quantile bound and $a_n = \frac{\delta}{\sqrt{n}}\Phi^{-1}(1-\frac{\delta}{2})$ for a two-sided quantile bound and take  $(\oneq,\twoq)$ to be the efficient quantile bounds $(q_n,q^d_n)$ of \cref{cpt-quantile} with the default auxiliary bounds of \cref{aux_q_bounds}.
\end{rem}

%% file: arxiv-sections/experiments.tex
\section{Numerical Evaluation}\label{sec:numerical}
We now turn to a numerical evaluation of our efficient bounds.  Python code implementing our bounds and reproducing all plots can be found at 
\begin{talign}
\textup{\href{https://github.com/lmackey/gauss\_conc/}{https://github.com/lmackey/gauss\_conc/}.}
\end{talign}

\subsection{Efficient quantile bounds}\label{sec:numerical2}
\begin{figure}[htb!]
\includegraphics[width=\textwidth]{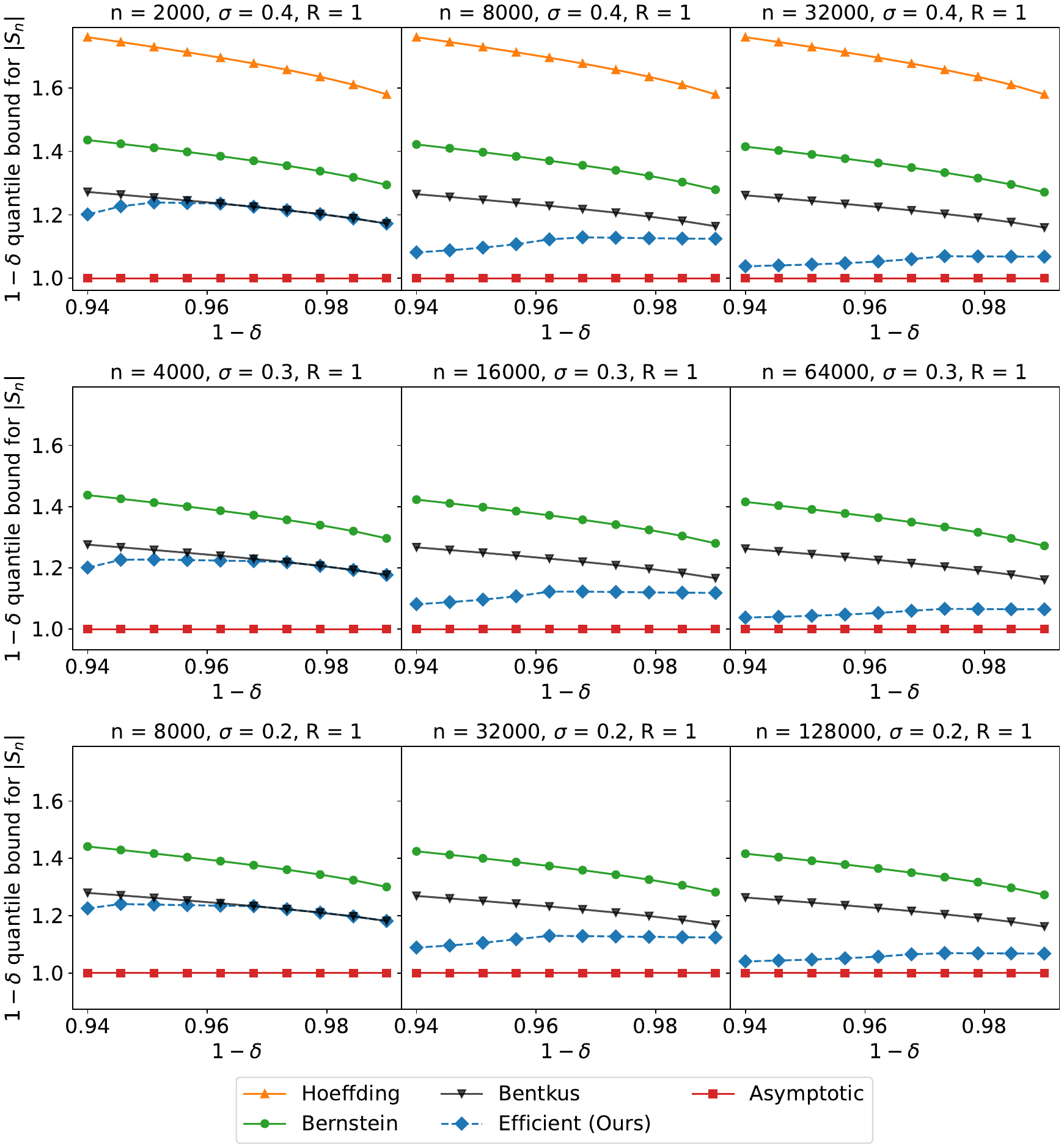}
\caption{\textbf{Quantile bounds for $\mbi{|S_n|}$} relative to the asymptotic width $\sigma\Phi^{-1}(1-\frac{\delta}{2})$ with fixed boundedness parameter $R=1$ and varying sample size $n$ and  variance $\sigma^2$. 
Unlike the inefficient Hoeffding, Bernstein, and Bentkus bounds, 
the efficient quantile bound of \cref{cpt-quantile,aux_q_bounds}  converges to the optimal asymptotic width as $n$ increases. 
\label{fig:quantile}}
\end{figure}
\begin{figure}[htb!]
\includegraphics[width=\textwidth]{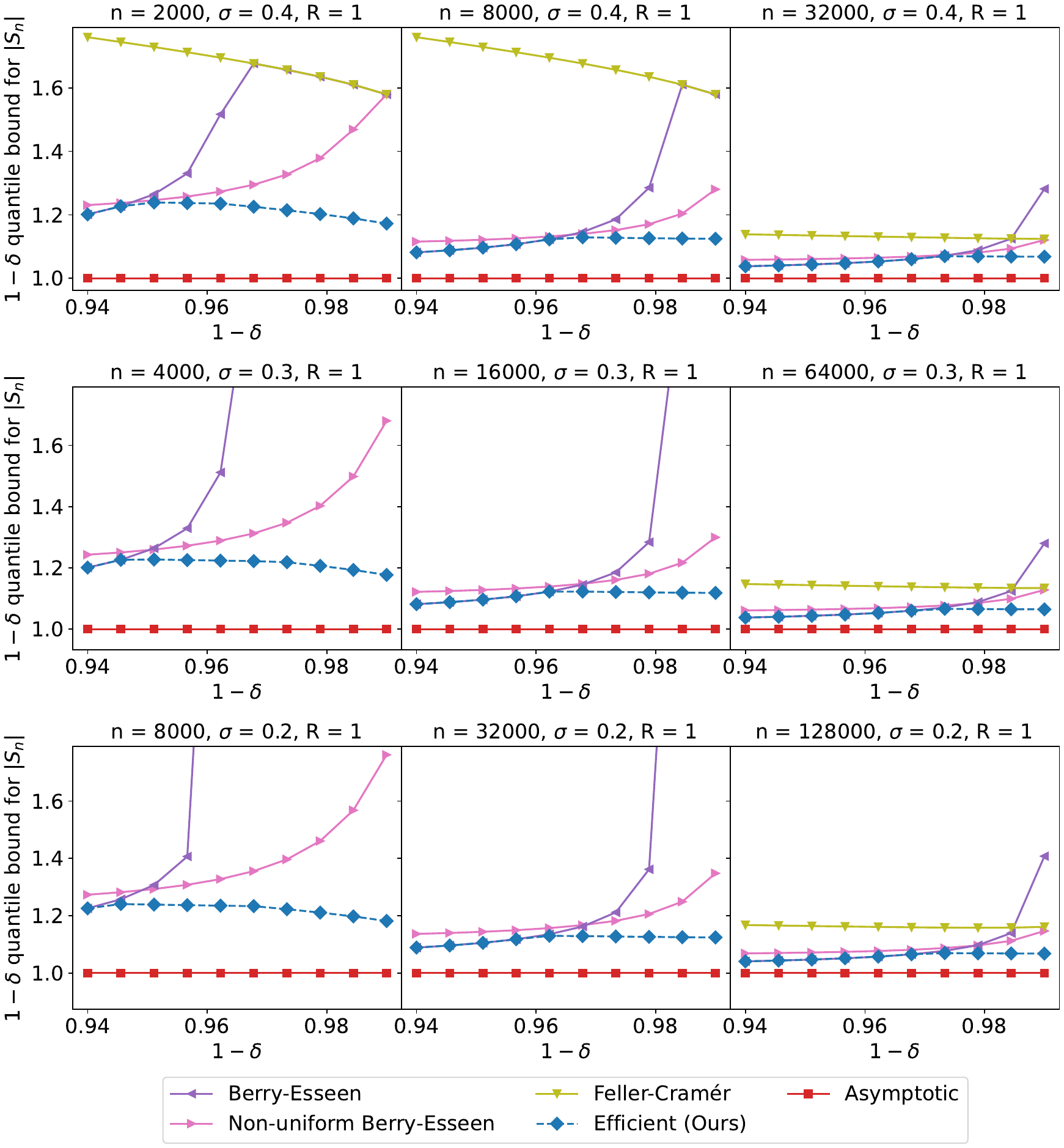}
\caption{\textbf{Efficient quantile bounds for $\mbi{|S_n|}$} relative to the asymptotic width $\sigma\Phi^{-1}(1-\frac{\delta}{2})$ with fixed boundedness parameter $R=1$ and varying sample size $n$ and  variance $\sigma^2$. 
The efficient quantile bound of \cref{cpt-quantile,aux_q_bounds} provides tighter estimates than the Berry-Esseen, non-uniform Berry-Esseen, and Feller-\Cramer bounds for higher confidence levels $1-\delta$.
\label{fig:quantile-efficient}}
\end{figure}

\cref{fig:quantile} compares the efficient two-sided quantile bound of \cref{cpt-quantile,aux_q_bounds} with the Bernstein quantile bound 
$P(|S_n| > \frac{\Rsig}{3\sqrt{n}}\log(2/\delta)+\sigma \sqrt{ 2 \log(2/\delta )})\le \delta$ \citep[Thm.~2.10]{boucheron2013concentration}
and the two-sided quantile bounds obtained by inverting the Hoeffding \cref{eq:bern_tails} and
Bentkus \cref{eq:bentkus_tails} tail bounds. 
As anticipated, the efficient bound converges to the optimal asymptotic width as $n$ increases, while the inefficient Bernstein, Hoeffding, and Bentkus bounds remain bounded away from the optimum for all $n$. 

\cref{fig:quantile-efficient} compares the efficient quantile bound of \cref{cpt-quantile,aux_q_bounds} to the two-sided quantile bounds obtained by inverting the 
Feller-\Cramer \cref{eq:feller-cramer}, 
Berry-Esseen (BE) \cref{eq:berryesseen}, and non-uniform BE \cref{eq:nonuniformberryesseen} tail bounds with 
$C_{R,\sig} = \min(.3328  ( \frac{\Rsig}{\sig} + .429 ),.33554 ( \frac{\Rsig}{\sig} + .415 ))$ \citep{shevtsova2011absolute}
and
$\tilde C_{R,\sig} = \min(17.36 \frac{\Rsig}{\sig} , 15.70\frac{\Rsig}{\sig} + 0.646) $ \citep[p.~54]{shevtsova2017absolute}.\footnote{The BE, non-uniform BE, and Feller-\Cramer bounds can be infinite, but we constrain them to be no larger than the always-valid Hoeffding bound.} 
The new efficient bound provides tighter estimates for higher confidence levels $1-\delta$, due to the improved underlying tail decay. %

\subsection{Application: Efficient empirical Berry-Esseen bounds}
\begin{figure}[htb!]
\includegraphics[width=\textwidth]{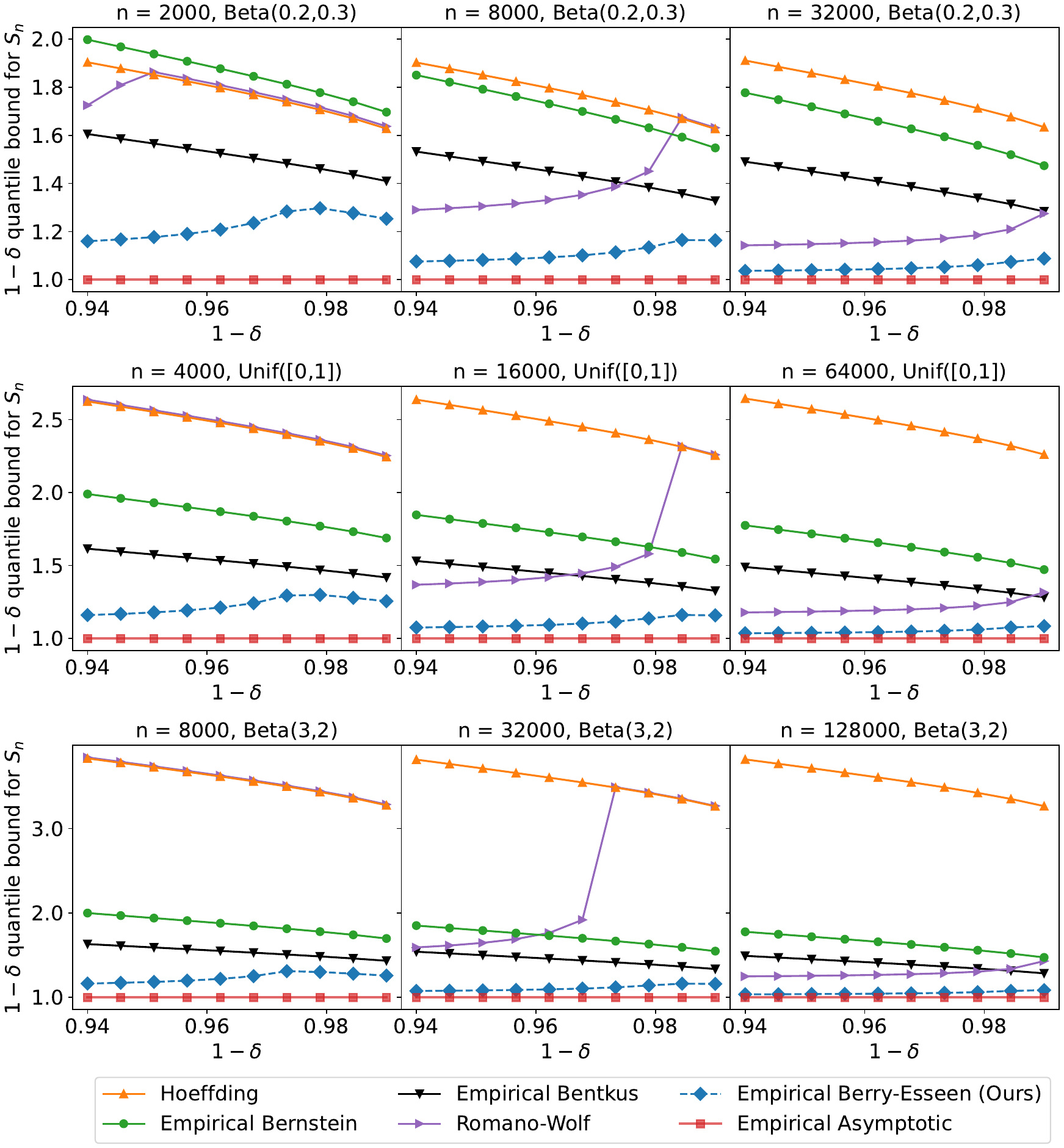}
\caption{\tbf{Empirical quantile bounds for $\mbi{S_n}$} relative to $\hat{\sigma}\Phi^{-1}(1-\delta)$ with fixed boundedness parameter $R=1$ and varying sample size $n$ and empirical variance $\hat{\sigma}^2$. 
The efficient empirical Berry-Esseen (EBE) bound of \cref{efficient-unknown,default-efficient-unknown} converges to the optimal asymptotic width as $n$ increases, while the empirical Bernstein, empirical Bentkus, and Hoeffding bounds remain bounded away from the optimum for all $n$. In addition, the efficient EBE bound provides tighter estimates than the efficient $I_{n,3}$ interval of \citet{romano2000finite}.\label{fig:emp_quantile}}
\end{figure}
\cref{fig:emp_quantile} compares the efficient empirical Berry-Esseen (EBE) bound of \cref{efficient-unknown,default-efficient-unknown} with the Hoeffding quantile bound,
$\P(S_n > R\sqrt{\log(1/\delta)/2}) \leq \delta$; 
the 
empirical Bernstein quantile bound \citep[Thm.~4]{maurer2009empirical} commonly deployed in reinforcement learning \citep{mnih2008empirical,audibert2009exploration}, 
\begin{talign}
\P\Big(S_n > \hat{\sigma} \sqrt{\log(2/\delta) \frac{n}{n-1}} + \frac{7}{3} R\log(2/\delta) \frac{\sqrt{n}}{n-1}\Big) \leq \delta
    \qtext{for}
\hat\sigma^2 
    \defeq 
\frac{1}{n}\sum_{i\le n}(W_i-\wbar)^2;
\end{talign}
an empirical Bentkus quantile bound based on \citet[(33)]{kuchibhotla2021near};
and the efficient $I_{n,3}$ quantile bound of \citet{romano2000finite},\footnote{The Romano-Wolf bound can be infinite, but we constrain it to be no larger than the Hoeffding $1-(\delta-\beta_n)$ quantile bound.} 
with its free parameter $\beta_n$ set to match our default value, $a_n = \frac{\delta}{\sqrt{n}}\Phi^{-1}(1-\delta)$.

As advertised, the efficient EBE bound converges to the optimal asymptotic width as $n$ increases, while the empirical Bernstein, empirical Bentkus, and Hoeffding bounds remain bounded away from the optimum for all $n$. 
Moreover, the efficient EBE bound provides a consistently tighter estimate than the Romano-Wolf bound.

%% file: arxiv-sections/mc.tex
\begin{figure}[htb!]
\includegraphics[width=\textwidth]{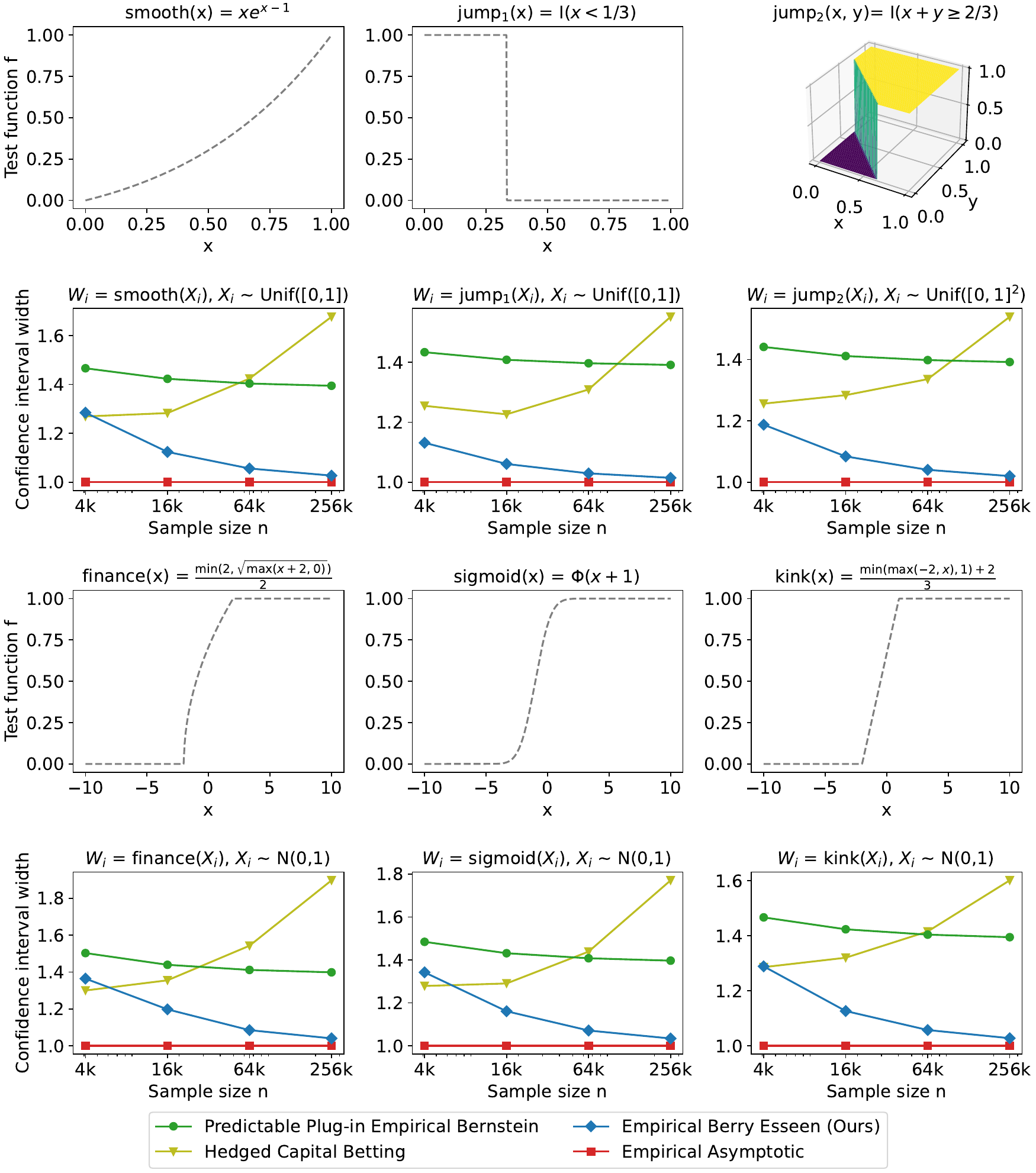}
\caption{\tbf{Widths of Monte Carlo confidence intervals for $\mbi{\E(f(X_1))}$} relative to the empirical asymptotic width ${\hat{\sigma}\Phi^{-1}(1-\frac{\delta}{2})}\frac{2}{\sqrt{n}}$ with  $\hat{\sigma}^2 = \frac{1}{n}\sum_{i=1}^n (W_i - \bar{W}_n)^2$ and confidence level $1-\delta=0.95$. 
We display relative widths averaged across $10$ independent replicates of the experiment. 
The efficient empirical Berry-Esseen intervals of \cref{efficient-unknown,default-efficient-unknown} converge to the optimal asymptotic size as $n$ increases, while the predictable plug-in empirical Bernstein and hedged capital betting bounds remain bounded away from the optimum for all $n$.\label{fig:ci}}
\end{figure}
\subsection{Application: Monte Carlo Confidence Intervals for Numerical Integration}\label{sec:mc}
A common use of concentration inequalities is in numerical integration, to provide a confidence interval for the expectation $\E(f(X_1))$ of a function $f$ using an \iid sample $(f(X_i))_{i=1}^n$. 
Recently, \citet{jain2025empirical} employed the state-of-the-art predictable plug-in empirical Bernstein (PrPl-EB) and hedged capital betting (HCB) intervals of \citet{waudby2024estimating} for this purpose, as they provide some of the narrowest confidence intervals known for bounded observations. 
However, the PrPl-EB intervals are also provably inefficient, as they asymptotically match the suboptimal width of Bernstein's inequality rather than the optimal width of the CLT~\citep[Sec.~E.3]{waudby2024estimating}.  
\cref{fig:ci} compares the PrPl-EB and HCB intervals with the efficient empirical Berry-Esseen (EBE) interval of \cref{efficient-unknown,default-efficient-unknown} 
for each of six benchmark test functions studied by \citet{jain2025empirical}.
In each case, we find that the PrPl-EB and HCB intervals remain bounded away from the optimal asymptotic width, while the EBE interval converges to the optimum, resulting in the tightest bounds for larger sample sizes.

%% file: arxiv-sections/discussion.tex
\section{Discussion and Related Work}\label{sec:discussion}
In this work, we have derived new computable tail and quantile bounds for the scaled deviations $S_n = \sqrt{n}(\wbar - \E(W_1))$ with asymptotically optimal size, finite-sample validity, and sub-Gaussian decay. 
These bounds enable the construction of efficient confidence intervals with correct coverage for any sample size.
Our concentration inequalities arise from new computable bounds on a non-uniform Kolmogorov distance  and the $p$-Wasserstein distances to a Gaussian.

The notion of efficient confidence intervals for the mean was introduced by \citet{romano2000finite}. In their Thm.~2.1, \citeauthor{romano2000finite}  showed that efficient---that is, asymptotically minimal-length when scaled by $\sqrt{n}$---confidence intervals must converge to the width of the asymptotic Gaussian intervals implied by the CLT \cref{eq:gaussian_tails}.
Moreover, Sec.~3 of \citeauthor{romano2000finite} surveys a number of procedures for constructing confidence intervals that are either finite-sample invalid (including the bootstrap \citep{efron1979bootstrap} and methods based on Edgeworth expansions \citep{hall2013bootstrap}) or inefficient (including the methods of \citet{anderson1969confidence} and \citet{gasko1991new}).
\citeauthor{romano2000finite} conclude by developing an efficient valid confidence interval for the mean of variables supported on $[0,1]$ but report that it is ``unfortunately, much too wide for a reasonable sample.''
Our new efficient bounds are developed in an entirely different manner, and, as we  demonstrate in \cref{sec:numerical}, improve upon the Romano-Wolf interval and the most commonly used empirical concentration inequalities. 

Our zero-bias coupling arguments generalize the uniform, $u$-independent bounds of \citet{chen2011normal} and \citet{ross2011fundamentals} to derive tighter non-uniform bounds with sub-Gaussian decay.
Our Wasserstein-bounding arguments build on the pioneering work of \citet{bonis2020stein} who derived Wasserstein convergence rates for the CLT with inexplicit constants. Our new arguments lead to tighter estimates of the distance to Gaussianity and explicit, practical constants.

Our results also suggest simple strategies for making any concentration inequality or confidence region efficient. For efficient concentration, one can simply take the minimum of any existing tail bound and  \cref{efficient-zero-tail} or \cref{cpt_concentration} to simultaneously reap the small-sample benefits of the former and the large-sample benefits of the latter.
For efficient confidence, one can divide the total confidence budget between an existing region and the efficient region of \cref{efficient-unknown} and then intersect the two regions.
The result will remain efficient if the budget allocated to the auxiliary region vanishes as $n$ grows.
These strategies are particularly relevant given the recent renewed interest in deriving tighter concentration inequalities for bounded random variables \citep[see, e.g.,][]{jun2019parameter,waudby2024estimating,orabona2023tight}.

In \cref{sec:mc}, we demonstrated the usefulness of our bounds by constructing tighter confidence intervals for Monte Carlo integration. A second natural application is to the construction of \emph{risk-controlling prediction sets}, that is, sets of predicted outcomes that are guaranteed to have high expected utility with high probability according to a given quality measure \citep{bates2021distribution}. \citet{bates2021distribution} reduce the problem of risk control to constructing confidence intervals for the unknown expected utility and employ the concentration inequalities of \citet{hoeffding1963probability}, \citet{bentkus2004hoeffding}, \citet{maurer2009empirical}, and \citet{waudby2024estimating} to form their prediction sets. A tighter confidence interval based on efficient concentration would yield more informative and hence more actionable prediction sets.

Finally, while our work has focused on concentration of the sample mean assuming boundedness, non-zero variance, and a sampling distribution independent of the sample size, we conjecture that our analyses can be adapted to 
(1) derive efficient, computable concentration inequalities for more general classes of asymptotically normal statistics or for self-normalized statistics, in the spirit of \citet{jing2003self-normalized} and \citet{pinelis2007exact}; 
(2) take advantage of other favorable properties of the $W_1$ distribution like symmetry or a vanishing third moment; 
and (3) account for the non-Gaussian limits that arise when the distribution underlying $(W_i)_{i=1}^n$ is allowed to vary with the sample size $n$.
We leave these important challenges for future work.

%% file: arxiv-sections/zero-proof.tex
\section{\pcref{zero-tail}}\label{proof-zero-tail}
Fix any $u\in\reals$.
Our proof structure, based on Stein's method, mimics that of \citet[Thm.~3.27]{ross2011fundamentals} but employs $u$-dependent bounds in place of the $u$-independent bounds invoked by \citeauthor{ross2011fundamentals}.

\paragraph*{Solving the Stein equation}
First we define the \emph{Stein equation} \citep{stein1986approximate}
\begin{talign}
f_u'(w) - w f_u(w) = \indic{w \leq u} - \Phi(u).
\end{talign}
By \citet[Lem.~2.2]{chen2011normal}, the absolutely continuous function
\begin{talign}\label{indic-stein-solution}
f_u(w) = \sqrt{2\pi} \exp(\frac{w^2}{2}) 
\Phi(w \wedge u) \Phi^c(w \vee u)
\end{talign}
solves the Stein equation, and we can therefore write
\begin{talign}
\P(\Sz \leq u) - \Phi(u) 
    &= 
\E[f_u'(\Sz) - \Sz f_u(\Sz)]
    =\E[Sf_u(S) - \Sz f_u(\Sz)]\ ,
\end{talign}
where the final equality uses the definition of the zero-biased distribution.
Negating both sides, we obtain
\begin{talign}
\P(\Sz > u) - \Phi^c(u) 
    &= 
\E[\Sz f_u(\Sz) - Sf_u(S)].
\end{talign}

\paragraph*{Bounding the Stein solution}
Now let $g_u(w) = w f_u(w)$.
By \citet[Eq.~(2.81)]{chen2011normal}, the function 
\begin{talign}
\label{hu-def-full}
h_u(w)
    &\defeq 
\begin{cases}
\sqrt{2\pi} \Phi^c(u) ( (1+w^2) \exp(w^2/2) \Phi(w) + w / \sqrt{2\pi} )
    &
\text{if } w \leq u \\
\sqrt{2\pi} \Phi(u) ( (1+w^2) \exp(w^2/2) \Phi^c(w) - w / \sqrt{2\pi} )
    & 
\text{if } w > u
\end{cases} \\
    &=
w (\Phi^c(u) - \indic{w > u}) + (1+w^2) f_u(w)
\\
    &=
w(\indic{w \leq u} \Phi^c(u) 
- \indic{w > u} \Phi(u))
+ (1+w^2) f_u(w)
\end{talign}
matches the derivative of $g_u(w)$ whenever $w\neq u$.
Since the absolute continuity of $f_u$ implies that $g_u$ is absolutely continuous on compact intervals, the fundamental theorem of calculus yields 
\begin{talign}
\E[g_u(\Sz) - g_u(S)] 
    = 
\E[\int_{0}^1 h_u(\Sz+x(S-\Sz)) (\Sz-S)dx]
    =
\E[h_u(\Smid)(\Sz-S)].
\end{talign} 

Our next result, proved in \cref{proof-hu-growth} provides a suitable $u$-dependent bound on $h_u$.

\begin{lemma}[Growth of $h_u$]\label{hu-growth}
For any $u \geq 0, \lam \in [0,1],$ and $w\in\reals$,
\begin{talign}
0 
    &\leq 
h_u(w) 
    \leq 
h_u(\lam u)
    +
(h_u(u) - h_u(\lam u)) \indic{w > \lam u}
	-
u \indic{w > u}.
\end{talign}
for $h_u$ defined in \cref{hu-def-full}.
\end{lemma}

\paragraph*{$\Sz$ tail bound}
Fix any $\lam \in [0,1]$. 
\cref{hu-growth}, our almost sure assumption $\Sz - S \leq \delta$, and the definition of $\Smid$ together imply that
\begin{talign}
&\P(\Sz > u) - \Phi^c(u) 
    \leq
\delta\, \E[h_u(\Smid)] \\
    &\leq
\delta\, 
[h_u(\lam u)
    + 
(h_u(u) - h_u(\lam u)) \P(\Smid > \lam u)
	-
 u\, \P(\Smid > u)].
\label{Sz-tail-bound}
\end{talign}

\paragraph*{From $S$ tails to $\Sz$ tails}
When $S - \Sz \leq \delta$ almost surely, we additionally have
\begin{talign}\label{S-tails-to-Sz-tails}
S - \Smid
    &\leq 
(S - \Sz) (1-U)
    \leq
\delta
    \qtext{and therefore} \\
\P(S > u + \delta) 
    &= 
\P(\Sz\wedge\Smid > u + \delta + \Sz\wedge\Smid-S) \\
    &\leq 
\P(\Sz\wedge\Smid > u)
    \leq
\P(\Sz > u)\wedge  \P(\Smid > u).
\end{talign}
Combining this inequality with our $\Sz$ tail bound \cref{Sz-tail-bound} yields
\begin{talign}&
\P(S > u + \delta)
    \leq
\Phi^c(u)
    +
\delta\, 
[h_u(\lam u)
    + 
(h_u(u) - h_u(\lam u)) \P(\Smid > \lam u)
	-
 u\, \P(S > u + \delta)].
\end{talign}
Rearranging the terms of this expression yields the final advertised result.

\subsection{Proof of \cref{hu-growth}: Growth of $h_u$}\label{proof-hu-growth}
Our proof relies on the following $\Phi^c$ property.

\begin{lemma}[Growth of $\Phi^c$]\label{phic-growth}
The following function is decreasing for $w \geq 0$:
\begin{talign}
b(w) 
    \defeq
\sqrt{2\pi} (1+w^2) \exp(w^2/2) \Phi^c(w) - w.
\end{talign}
\end{lemma}
\begin{proof}
Fix any $w\geq 0$. 
The derivative of $b$ takes the form
\begin{talign}
b'(w)
    =
\Phi^c(w)(3w+w^3)/\varphi(w) - (2+w^2).
\end{talign}
Since the Mills ratio $\frac{\Phi^c(w)}{\varphi(w)} < \frac{2+w^2}{3w+w^3}$ \citep[Thm.~2.1]{lee1992laplace}, $b'(w)$ is negative and $b$ is decreasing. 
\end{proof}

Fix any $u\geq 0, \lam \in [0,1],$ and $w \in \reals$. 
We divide our proof into cases based on $w$.

\paragraph*{Lower bound, $w\neq u$:}
Since $g_u$ is differentiable for $w\neq u$ and increasing by \citet[Lem.~2.3]{chen2011normal}, $h_u(w) = g_u'(w) \geq 0$.

\paragraph*{Lower bound, $w = u$:}
We have $h_u(u) = \lim_{v\uparrow u} h_u(v) \geq 0$.

\paragraph*{Upper bound, $w > u$:}
Since $b$ is decreasing (\cref{phic-growth}), the definition of $f_u$ \cref{indic-stein-solution} implies that
\begin{talign}
(1+w^2) f_u(w) \indic{w > u}
    &=
\sqrt{2\pi} (1+w^2) \exp(w^2/2) \Phi^c(w) \Phi(u) \indic{w > u} \\
    &\leq
(b(u) + w) \Phi(u) \indic{w > u}.
\end{talign}
Hence, 
\begin{talign}
h_u(w) \indic{w > u}
    &=
(1+w^2) f_u(w) \indic{w > u} - w\Phi(u) \indic{w > u}
    \leq
b(u) \Phi(u) \indic{w > u}.
\end{talign}

\paragraph*{Upper bound, $\lam u \geq w \geq 0$:}
Since $g_u$ is increasing \citep[Lem.~2.3]{chen2011normal}, $f_u$ is increasing for $w < u$, and $\lam u f_u(\lam u) + \Phi^c(u)$ is nonnegative, we have
\begin{talign}
h_u(w) \indic{\lam u \geq w \geq 0}
    &=
(f_u(w) + w (w f_u(w) + \Phi^c(u) ) ) \indic{\lam u \geq w \geq 0} \\
    &\leq 
[f_u(\lam u) + w (\lam u f_u(\lam u) + \Phi^c(u))] \indic{\lam u \geq w \geq 0}\\
    &\leq 
[f_u(\lam u) + \lam u (\lam u f_u(\lam u) + \Phi^c(u))] \indic{\lam u \geq w \geq 0}\\
    &= 
h_u(\lam u) \indic{\lam u \geq w \geq 0}
    =
h_u(\lam u) (\indic{w \geq 0} - \indic{w > \lam u}).
\end{talign}

\paragraph*{Upper bound, $w < 0$:}
Since $w < 0 \leq u$, we use the definition of $f_u$ \cref{indic-stein-solution}, the fact that $b$ is decreasing (\cref{phic-growth}), and the nonnegativity of $\Phi^c$ to derive
\begin{talign}
(1+w^2) f_u(w) \indic{w < 0}
    &=
\sqrt{2\pi} \exp(w^2/2) \Phi^c(|w|) \Phi^c(u) \indic{w < 0} \\
    &=
(b(|w|) + |w|) \Phi^c(u) \indic{w < 0} \\
    &\leq
(b(0) + |w|) \Phi^c(u) \indic{w < 0}.
\end{talign}
Since $\lam u \geq 0$ by assumption, our prior derivation implies $h_u(\lam u) \geq h_u(0)$ and hence,
\begin{talign}
h_u(w) \indic{w < 0} 
    &= 
( (b(0) + |w|) - |w| ) \Phi^c(u) \indic{w < 0} \\
    &=
b(0) \Phi^c(u) \indic{w < 0} 
    =
\sqrt{\frac{\pi}{2}} \Phi^c(u) \indic{w < 0} \\
    &=
h_u(0) \indic{w < 0}
    \leq
h_u(\lam u) \indic{w < 0}.
\end{talign}

\paragraph*{Upper bound, $u \geq w > \lam u$:}
Since $u,w, f_u(u),$ and $\Phi^c(u)$ are nonnegative,  $g_u$ is increasing \citep[Lem.~2.3]{chen2011normal}, 
and $u \geq w$, we have 
\begin{talign}
h_u(w) \indic{u \geq w > \lam u} 
    &=
(f_u(w) + w (w f_u(w) + \Phi^c(u) ) ) \indic{u \geq w > \lam u} \\
    &\leq 
(f_u(u) + w (u f_u(u) + \Phi^c(u) ) ) \indic{u \geq w > \lam u} \\
    &\leq
(f_u(u) + u (u f_u(u) + \Phi^c(u) ) ) \indic{u \geq w > \lam u} \\
    &= 
h_u(u) \indic{u \geq w > \lam u} 
    =
h_u(u) (\indic{w > \lam u}  - \indic{w > u}).
\end{talign}

\paragraph*{Complete upper bound:}
Taken together, our upper bounds yield
\begin{talign}
h_u(w)
    &\leq 
b(u) \Phi(u) \indic{w > u}
	+ 
h_u(u) (\indic{w > \lam u}  - \indic{w > u})
	\\&\qquad+ 
h_u(\lam u) (\indic{w \geq 0} - \indic{w > \lam u})
	+ 
h_u(\lam u) \indic{w < 0} \\
	&\leq 
(b(u) \Phi(u) - h_u(u)) \indic{w > u}
	+ 
(h_u(u) - h_u(\lam u)) \indic{w > \lam u}
	+ 
h_u(\lam u) \\
	&= 
(h_u(u) - h_u(\lam u)) \indic{w > \lam u}
	+ 
h_u(\lam u)
    -
u\indic{w > u}.
\end{talign}

%% file: arxiv-sections/wasserstein-proof.tex
\section{\pcref{lemm:dino_cpt}}
\label{proof-lemm:dino_cpt}
We begin by defining some convenient shorthand notation. 
For each $p\geq 1$ and $q\in[0,1]$, we define $\norm{\Bin(n,q)}_p\defeq\norm{V}_p$ for $V\sim\Bin(n,q)$ and 
make use of the constants
\begin{talign}
\widetilde \Rsig\defeq \Rsig/\sig = \frac{1}{2}\tilde R+\frac{1}{2}\sqrt{\tilde R^2-4},
\end{talign}
\begin{talign}
A_p\defeq \frac{\sqrt{e}\sqrt{p+2}(2e)^{1/p}}{\sqrt{2}},\qquad  A^*_{n,p}\defeq \frac{({p+2})n^{1/p}}{2\sqrt{n}},\qquad \tilde A_{n,p}\defeq A^*_{n,p}\tilde R^{-2/p},
\end{talign}
\begin{talign}
U_{n,p}\defeq A_p+\tilde R \tilde A_{n,p},\qquad   \tilde U_{n,p}\defeq\sqrt{2} A_p+2^{1/p}\tilde R \tilde A_{n,p}, \qtext{and}
\end{talign}
\begin{talign}\label{eq:Cnp}
   C_{n,p}\defeq 
\min\begin{cases}\tilde A_{n,p}\tilde R 2^{1/p}+\sqrt{2}A_{p}\\{\sqrt{2}}\Big(\frac{2\Gamma(\frac{p+1}{2})}{\sqrt{\pi}}\Big)^{1/p}\begin{cases}\tilde R^{1-2/p}\text{\quad if }p<4\\\tilde R^{1-2/p} \wedge \frac{\tilde R}{\sqrt{n}}\sqrt{\|\Bin(n,\frac{2}{\tilde R^2})\|_{p/2}}\quad \text{if } p \ge 4.\end{cases} 
\end{cases}\!\!\!\!\!\!\!\!
\end{talign}

We will focus principally on establishing the bound $\frac{1}{\sig}\wass(S_n,\Gsn(0,\sigma^2))
\le 
\wassbd$
and derive the secondary bound 
$\omega^{R,\kappa,1}_p(\sig)\leq \omega_p^{R,\kappa,2}(\sig)$ through a series of asides demarcated by a vertical bar on the left-hand side. 
For the case of $p=1$, we invoke \citet[Cor.~4.2]{chen2011normal} and \cref{summand-deviation-bound} in turn to find that
\begin{talign}
\frac{1}{\sig}\mathcal{W}_1(S_n,\Gsn(0,\sigma^2))\le \frac{\E[|W_1-\E(W_1)|^3]}{\sig^3\sqrt{n}} \le \frac{\Rsig}{\sig \sqrt{n}} =  \wassbd.
\end{talign}
When $p> n+1$ the triangle inequality and \cref{dino_veg,mal_aux_ventre} imply that
\begin{talign}
\sigma^{-1}\mathcal{W}_p(S_n,\Gsn(0,\sigma^2))\le \|Z\|_p+\|S_n/\sigma\|_p
\le \sqrt{p-1}+\sqrt{p-1}\Rsig/\sigma.
\end{talign}

Now fix any $p \in [2,n+1]$, define $\tilde S_n\defeq S_n/\sig$, and let $Z\sim\Gsn(0,1)$ be independent of $(W_i)_{i\ge1}$.
For ease of notation, 
we will write $\wass(\tilde S_n,Z)$ in place of $\wass(\tilde S_n,\Gsn(0,1))$ and 
use $(X_i)_{i\geq 1}$ to represent the centered and rescaled random variables 
\begin{talign}
X_i\defeq\frac{W_i-\E[W_i]}{\sigma\sqrt{n}},
\end{talign} 
which satisfy $\tilde S_n = \sum_{i\le n}X_i$, $\Var(X_1)=\frac{1}{n}$, and  $\|X_1\|_{\infty}\le {\widetilde\Rsig}{/\sqrt{n}}$ by \cref{summand-deviation-bound}. 
Hence, for all $k\ge 0$, the moments of $X_1^k$ can be upper bounded as
\begin{talign}\label{moment_eq_t}&\|X_1^k\|_p=\big(\E[|X_1|^{kp}]\big)^{1/p}\le \big(\|X_1\|_{\infty}^{kp-2}\E[X_1^2]\big)^{1/p} \le\frac{\widetilde \Rsig^k}{\widetilde \Rsig^{2/p}\sqrt{n}^k}.
\end{talign}%

Consider a random index $I\sim \Unif(\{1,\dots,n\})$ and a sequence $(X_i')_{i\geq1}\eqdist (X_i)_{i\geq1}$ with $(I,(X_i')_{i\geq1},(X_i)_{i\geq1})$ mutually independent.
Define an exchangeable copy of $\tilde S_n$,
\begin{talign}
S'_{n}\defeq \tilde S_n+(X'_I-X_I),
\end{talign} 
and the exchangeable pair difference,
\begin{talign}
Y\defeq\tilde S_n-S'_{n}.
\end{talign} 
Let $h_k(x)\defeq e^{x^2/2} \frac{\partial^k}{\partial x}e^{-x^2/2}$ designate the $k$-th Hermite polynomial and define  $H_k\defeq h_k(Z)$.
Finally, fix any $\kappa> \frac{\tilde R^2}{n}$ and define
\begin{talign}
\mnkappa\defeq {\sqrt{1-\frac{\tilde R^2}{n\kappa}}}.%
\end{talign}
A slight modification of  \citet[Thm.~3]{bonis2020stein} shows that
\begin{talign}\label{sneeze}\wass (\tilde S_n,Z)\le &\int^{-\frac{1}{2}\log(1-\frac{\tilde R^2}{n\kappa})}_0\|e^{-t}\tilde S_n-\frac{e^{-2t}}{\sqrt{1-e^{-2t}}}Z\|_p\dt \\&+\int_{-\frac{1}{2}\log(1-\frac{R^2}{n\kappa})}^{\infty}e^{-t}\|n\E[Y\mid \tilde S_n]-
\tilde S_n\|_p\dt\\&+\int_{-\frac{1}{2}\log(1-\frac{\tilde R^2}{n\kappa})}^{\infty}\frac{e^{-2t}\|H_1\|_p}{\sqrt{1-e^{-2t}}}\|\frac{n}{2}\E[Y^2\mid \tilde S_n]-1\|_p\dt
\\&+\sum_{k\ge 3}\int_{-\frac{1}{2}\log(1-\frac{\tilde R^2}{n\kappa})}^{\infty}\frac{e^{-kt}\|H_{k-1}\|_p}{k!(\sqrt{1-e^{-2t}})^{k-1}}n\|\E[Y^k\mid \tilde S_n]\|_p\dt
\\
\defeq &\ (a_0)+(a_1)+(a_2)+(a_3).
\end{talign}

\subsection*{Bounding $(a_0)$}
We first bound $(a_0).$ %
To this end, fix $\epsilon>0$ and select $G\sim \mathcal{N}(0,1)$ independent from $Z$ such that $\|\tilde S_n-G\|_p\le \mathcal{W}_p(\tilde S_n,Z)+\epsilon.$ 
By the triangle inequality, 
\begin{talign}
   (a_0)
   &\le\int^{-\frac{1}{2}\log(1-\frac{\tilde R^2}{n\kappa})}_{0}e^{-t}\|\tilde S_n-G\|_p+\|e^{-t}G-\frac{e^{-2t}}{\sqrt{1-e^{-2t}}}Z\|_p \dt
    \\&\overset{(a)}{=} (1-\mnkappa)\|\tilde S_n-G\|_p+\|Z\|_p\int^{-\frac{1}{2}\log(1-\frac{\tilde R^2}{n\kappa})}_{0}\frac{e^{-t}}{\sqrt{1-e^{-2t}}}\dt  \\&\le (1-\mnkappa)\big(\mathcal{W}_p(\tilde S_n,Z)+\epsilon\big) +\|Z\|_p\big(\frac{\pi}{2}-\sin^{-1}(\mnkappa)\big)
\end{talign} where (a) follows as $e^{-t}G-\frac{e^{-2t}}{\sqrt{1-e^{-2t}}}Z\overset{d}{=} \frac{e^{-t}}{\sqrt{1-e^{-2t}}}Z$.
Since $\epsilon>0$ was arbitrary, we have \begin{talign}&
 (a_0)
 \le (1-\mnkappa)\mathcal{W}_p(\tilde S_n,Z)+\|Z\|_p\big(\frac{\pi}{2}-\sin^{-1}(\mnkappa)\big).
\end{talign} 
In addition, if $p=2$, by independence of $Z$ and $S_n$ we obtain that 
\begin{talign}\label{eq:02}
   (a_0)
   &\le\int^{-\frac{1}{2}\log(1-\frac{\tilde R^2}{n\kappa})}_0\sqrt{e^{-2t}\|\tilde S_n\|_2^2+\frac{e^{-4t}}{{1-e^{-2t}}}\|Z\|_2^2} \dt
    \\&\le\int^{-\frac{1}{2}\log(1-\frac{\tilde R^2}{n\kappa})}_0\frac{e^{-t}}{\sqrt{1-e^{-2t}}}\dt
    \le
 \frac{\pi}{2}-\sin^{-1}(\mnkappa).
\end{talign} 
These results together with the bound \cref{sneeze} yield the inequality
\begin{talign}
\wass (\tilde S_n,Z)\le & \Big[\|Z\|_p\Big(\frac{\pi}{2}-\sin^{-1}(\sqrt{1-\frac{\tilde R^2}{n\kappa}})\Big)+(a_1)+(a_2)+(a_3)\Big]
\begin{cases}
\frac{1}{\mnkappa} & \text{if } p > 2 \\
1 & \text{if } p = 2.
\end{cases}
\end{talign} 

\subsection*{Bounding $(a_1)$}
 We will next bound $(a_1)$ of \cref{sneeze}. 
 Note that, since $I\sim\Unif(\{1,\dots,n\})$, 
\begin{talign}\label{merur2_t}
\E[Y\mid \tilde S_n]=\frac{1}{n}\sum_{i\le n}\E[X_i-X'_i\mid \tilde S_n]=\frac{1}{n}\sum_{i\le n}X_i-\E[X_i]=\frac{1}{n}\tilde S_n,
\end{talign} 
and hence $\|n\E[Y\mid \tilde S_n]-\tilde S_n\|_p=0$. 
Therefore $(a_1)=0$.

\subsection*{Bounding $(a_2)$}
We now turn to bounding $(a_2)$ of \cref{sneeze}. 
By Jensen's inequality, 
\begin{talign}\|\frac{n}{2}\E[Y^2\mid \tilde S_n]-1\|_p
&\le \|(\frac{1}{2}\sum_{i\le n}\E[X_i^2]+X_i^2)-1\|_p
=\frac{1}{2} \|\sum_{i\le n}X_i^2-1\|_p.
\end{talign} 
Moreover, by \cref{summand-deviation-bound},  $|X_i^2-1/n|\overset{a.s.}{\le}\max(\frac{1}{n}\widetilde \Rsigsqd-\frac{1}{n},\frac{1}{n}\big)$. 
First suppose $\frac{1}{n}\ge \frac{\widetilde \Rsigsqd-1}{n}$.
Since $\E[X_i^2]=1/n$, we know $\|X_i^2-1/n\|_p$  is maximized when  $\frac{n}{\widetilde \Rsigsqd} X_i^2\sim \Ber(\frac{1}{\widetilde \Rsigsqd}).$
Hence
\begin{talign}
\|X_i^2-1/n\|_p
\le  
\frac{1}{n}\Big[\frac{ \widetilde \Rsigsqd-1}{ \widetilde \Rsigsqd}\Big]^{1/p}\big( (\widetilde \Rsigsqd-1)^{p-1}+1\big)^{1/p}.
\end{talign} 
Now suppose $ \widetilde \Rsigsqd-1\geq 1$. We instead obtain that 
\begin{talign}
\|X_i^2-1/n\|_p
\le  
(\|X_i^2-1/n\|_\infty^{p-2} \|X_i^2-1/n\|_2^2)^{1/p}
\le
\frac{1}{n}\big( \widetilde \Rsigsqd-1\big)^{1-1/p}.
\end{talign} 
Therefore using the Marcinkiewicz-Zygmund inequality (\cref{dino_veg}) and the Rosenthal inequality (\cref{dino_veg2}) we find that
 \begin{talign}\|\frac{n}{2}\E[Y^2\mid \tilde S_n]-1\|_p
\overset{}{\le}\frac{1}{2\sqrt{n}}\min\begin{cases}\sqrt{p-1}(\widetilde \Rsigsqd-1)^{1-1/p}\\(\widetilde \Rsigsqd-1)^{1-1/p} A^*_{n,p}+\sqrt{\widetilde \Rsigsqd-1}A_p.
\end{cases}%
\end{talign}
for  $ \widetilde \Rsigsqd\ge 2$ and
\begin{talign}\|\frac{n}{2}\E[Y^2\mid \tilde S_n]-1\|_p
\overset{}{\le}\frac{1}{2\sqrt{n}}\min\begin{cases}\sqrt{p-1}\Big[\frac{ \widetilde \Rsigsqd-1}{ \widetilde \Rsigsqd}\Big]^{1/p}\big( (\widetilde \Rsigsqd-1)^{p-1}+1\big)^{1/p}\\\Big[\frac{ \widetilde \Rsigsqd-1}{ \widetilde \Rsigsqd}\Big]^{1/p}\big( (\widetilde \Rsigsqd-1)^{p-1}+1\big)^{1/p}A^*_{n,p}+\sqrt{\widetilde \Rsigsqd-1}A_p
\end{cases}%
\end{talign}
for $\widetilde \Rsigsqd\le 2$.
Alternatively, by \citet[Eq.~(2.8)]{esseen1975bounds} for $p<4$ and  \citet[Thm.~2.6]{cox1983sharp} for $p\ge 4$, we have the symmetrized estimate
\begin{talign}
\|\sum_{i\le n}X_i^2-1\|_p\le2^{-1/p} \|\sum_{i\le n}X_i^2-(X'_i)^2\|_p.
\end{talign}
Since the random variables $\big(X_i^2-(X'_i)^2\big)_{i\geq 1}$ are symmetric, with
\begin{talign}
\E[(X_i^2-(X'_i)^2)^2]
\le \frac{2}{n^2}(\widetilde \Rsigsqd-1)
\qtext{and} \E[(X_i^2-(X'_i)^2)^p]\le \frac{2}{n^{p}}(\widetilde \Rsigsqd-1)\widetilde \Rsig^{2(p-2)},
\end{talign}
an improvement on the  Marcinkiewicz-Zygmund inequality for symmetric random variables (\cref{dino_veg_symm}) implies
\begin{talign}
\|\sum_{i\le n}X_i^2-(X'_i)^2\|_p
   \le
\begin{cases}
\frac{\sqrt{2}2^{1/p}}{\sqrt{n}} \Big(\frac{\Gamma(\frac{p+1}{2})}{\sqrt{\pi}}\Big)^{1/p}\widetilde \Rsig^{2(1-2/p)}(\widetilde \Rsigsqd-1)^{1/p} & \stext{if} p\ge2\\
{\sqrt{2}}\Big(\frac{\Gamma(\frac{p+1}{2})}{\sqrt{\pi}}\Big)^{1/p}\frac{\tilde R^2}{n}\|\Bin(n,\frac{2(\widetilde \Rsigsqd-1)}{\widetilde \Rsig^4})\|_p &\stext{if} p\ge4.
\end{cases}
\end{talign}
Hence we finally obtain that%
\begin{talign}\label{highschool}&\qquad\|\frac{n}{2}\E[Y^2\mid \tilde S_n]-1\|_p
\\&\overset{}{\le}D_{n,p}\defeq\frac{1}{2\sqrt{n}}\min
\begin{cases}
\sqrt{p-1}\max\Big((\widetilde \Rsigsqd-1)^{1-1/p},\Big[\frac{ (\widetilde \Rsigsqd-1)^{p}+\widetilde \Rsigsqd-1}{ \widetilde \Rsigsqd}\Big]^{1/p}\Big)\\
\max\Big((\widetilde \Rsigsqd-1)^{1-1/p} ,\Big[\frac{ (\widetilde \Rsigsqd-1)^{p}+\widetilde \Rsigsqd-1}{ \widetilde \Rsigsqd}\Big]^{1/p}\Big)A^*_{n,p}
+\sqrt{\widetilde \Rsigsqd-1}A_p\\\sqrt{2}\Big(\frac{\Gamma(\frac{p+1}{2})}{\sqrt{\pi}}\Big)^{1/p}\widetilde \Rsig^{2(1-2/p)}(\widetilde \Rsigsqd-1)^{1/p} \\\sqrt{2}2^{-1/p}\Big(\frac{\Gamma(\frac{p+1}{2})}{\sqrt{\pi}}\Big)^{1/p}\frac{\tilde R^2}{\sqrt{n}}\sqrt{\|\Bin(n,\frac{2(\widetilde \Rsigsqd-1)}{\widetilde \Rsig^4})\|_{p/2}}\quad\stext{if} p\ge4.\end{cases} 
\end{talign}
This implies that 
{\begin{talign}\label{snow2_t}
    &\int_{-\frac{1}{2}\log(1-\frac{R^2}{n\kappa})}^{\infty}\|\frac{n}{2}\E[Y_t^2\mid \tilde S_n]-1\|_p\frac{e^{-2t}}{\sqrt{1-e^{-2t}}}\dt
    \\&\le D_{n,p}\int_{-\frac{1}{2}\log(1-\frac{R^2}{n\kappa})}^{\infty}\frac{e^{-2t}}{\sqrt{1-e^{-2t}}}\dt
     = D_{n,p}\int^{\sqrt{1-\frac{R^2}{n\kappa}}}_{0}\frac{t}{\sqrt{1-t^2}}\dt
    = \mnkappa^2 D_{n,p}.
\end{talign}}
\begin{proofaside}
The definition \cref{highschool}
also implies the further upper bound \begin{talign}\label{snow2_t_s}
    D_{n,p}\le\frac{1}{2\sqrt{n}}(\max(\widetilde \Rsigsqd-1,1))^{1-1/p}\tilde A^*_{n,p}+\sqrt{\widetilde \Rsigsqd-1}A_p  .
\end{talign}
\end{proofaside}
Therefore since $H_1(Z)=Z$ we obtain 
{\begin{talign}\label{fr_al_2}
(a_2)&\le (b_{2,1}^{\kappa,p,\tilde R})
\defeq
\|Z\|_p D_{n,p} M_{n,\kappa}^2.
\end{talign}}

\begin{proofaside}
    Since $\|Z\|_p\le \sqrt{p-1}$ by \cref{mal_aux_ventre}, we have the further upper bound
\begin{talign}\label{fr_al_2b}
 (b_{2,1}^{\kappa,p,\tilde R})\le (b_{2,2}^{\kappa,p,\tilde R})\defeq&\frac{\sqrt{p-1}}{2\sqrt{n}}\Big[(\max(\widetilde \Rsigsqd-1,1))^{1-1/p}A^*_{n,p}+\sqrt{\widetilde \Rsigsqd-1}A_p \Big] .
\end{talign}
\end{proofaside}

\subsection*{Bounding $(a_3)$}
Finally, we turn to bounding $(a_3)$ of \cref{sneeze}.  %
Since
\begin{talign}
   n \E[Y^k\mid X_1,\dots, X_n]=\sum_{i\le n}\E[(X_i-X'_i)^k\mid X_1,\dots, X_n]
   \qtext{for all}
   k \geq 1,
\end{talign} 
by Jensen's inequality, we have 
\begin{talign}\label{eq:EYk-bound}
n\big\|\E[Y^k\mid S_n]\big\|_p\le\big\|\sum_{i\le n}(X_i-X'_i)^k\big\|_p\qtext{for all} k\ge 1.
\end{talign}
We will derive different bounds for odd and even $k$, so we begin by writing
\begin{talign}
    (a_3)&\le \sum_{\substack{k\ge 3:\\k ~\textup{is~odd}}}\int_{-\frac{1}{2}\log(1-\frac{\tilde R^2}{n\kappa})}^{\infty}\frac{e^{-kt}\|H_{k-1}\|_p}{k!(\sqrt{1-e^{-2t}})^{k-1}}n\|\E[Y^k\mid \tilde S_n]\|_p\dt
    \\&+ \sum_{\substack{k\ge 4:\\k ~\textup{is~even}}}\int_{-\frac{1}{2}\log(1-\frac{\tilde R^2}{n\kappa})}^{\infty}\frac{e^{-kt}\|H_{k-1}\|_p}{k!(\sqrt{1-e^{-2t}})^{k-1}}n\|\E[Y^k\mid \tilde S_n]\|_p\dt
    \defeq(a_{3,1})+(a_{3,2}).
\end{talign} 

Let $k\ge 3$ be an \emph{odd} integer. The random variables $((X_i-X'_i)^k)_{i\ge1}$ are symmetric and therefore have a mean of zero. 
Moreover, if we define $W_i' \defeq \sig \sqrt{n}X_i' + \E[W_i]$, then 
\begin{talign}&\|(X_i-X'_i)^k\|_p=\frac{1}{\sqrt{n}^k\sigma^k}\||W_i-W'_i|^k\|_p
\\&\le \frac{1}{\sqrt{n}^k\sigma^k}\|\max(W_i,W_i')^{k-2/p}|W_i-W'_i|^{2/p}\|_p
\\&\le \frac{1}{\sqrt{n}^k\sigma^{2/p}}\tilde R^{k-2/p}\||W_i-W'_i|^{2/p}\|_p
\le \frac{2^{1/p}\tilde R^{k-2/p}}{\sqrt{n}^k}.
\end{talign}
Similarly, we also have \begin{talign}\|(X_i-X'_i)^k\|_2\le\frac{\sqrt{2}\tilde R^{k-1}}{\sqrt{n}^k}.\end{talign}
Therefore, \cref{dino_veg,dino_veg_symm,dino_veg2}
together imply
$\big\|\sum_{i\le n}(X_i-X'_i)^k\big\|_p\le 
\frac{ \tilde R^{k-1}}{\sqrt{n}^{k-1}}C_{n,p}$. 
Combining this with the inequality \cref{eq:EYk-bound} we obtain 
\begin{talign}(a_{3,1})&=\sum_{\substack{k\ge 3:\\k ~\textup{is~odd}}}\int_{-\frac{1}{2}\log(1-\frac{\tilde R^2}{n\kappa})}^{\infty}\frac{e^{-tk}\|H_{k-1}\|_p}{k!\sqrt{{1-e^{-2t}}}^{k-1}}n\|\E[Y^k\mid S_n]\|_p\dt
\\&\le\sum_{\substack{k\ge 3:\\k ~\textup{is~odd}}}
 {C_{n,p}}
 \int^{\sqrt{1-\frac{\tilde R^2}{n\kappa}}}_0\frac{x^{k-1}\tilde R^{k-1}\|H_{k-1}\|_p}{k!\sqrt{1-x^2}^{k-1}\sqrt{n}^{k-1}}\dx
 \\&\overset{(a)}{\le} \sum_{\substack{k\ge 1}}C_{n,p}
 \frac{\tilde R^{2k}\|H_{2k}\|_p}{n^k(2k+1)! }\int^{\sqrt{1-\frac{\tilde R^2}{n\kappa}}}_0\frac{x^{2k}}{(1-x^2)^k}\dx
 \\&\overset{(b)}{=}\frac{1}{2}\sum_{\substack{k\ge 1}}
C_{n,p} \frac{\tilde R^{2k}\|H_{2k}\|_p}{n(2k+1)!}\int_{\frac{\tilde R^2}{\kappa}}^n\frac{1}{\sqrt{y}}\big(\frac{1}{y}-\frac{1}{n}\big)^{k-\frac{1}{2}}\dy,
\end{talign}
where (a) is obtained by noting that all odd numbers $k$ can be written as $2m+1$ for an $m\in \mathbb{N}_+$ and (b) by the change of variables $y=n(1-x^2)$.
 To further upper bound the right-hand side, we  will invoke the Hermite polynomial moment bound (\cref{mal_aux_ventre})  $\|H_{k-1}\|_p\le {\sqrt{p-1}^{k-1}}{\sqrt{k-1!}}$ and use two applications of Stirling's approximation \citep{robbins1955remark} to conclude that,  for all  $m\in \mathbb{N}\setminus \{0\}$,
\begin{talign}\label{taylor_swift}
        \sqrt{(2m)!} &\geq \sqrt{\sqrt{2\pi(2m)} \cdot (2m/e)^{2m} \cdot \exp\left(\frac{1} {12(2m)+1}\right)} 
        \geq e^{-19/300} 2^m m! / (\pi m)^{1/4}.
    \end{talign}
These estimates imply that, for all $K_p\in \mathbb{N}_+$,
\begin{talign}
  (a_{3,1})&\le \frac{C_{n,p}}{2}\Big\{\sum_{\substack{1\le k\le K_p-1}}
 \frac{\tilde R^{2k}\|H_{2k}\|_p}{n(2k+1)!}\int_{\frac{\tilde R^2}{\kappa}}^n\frac{1}{\sqrt{y}}\big(\frac{1}{y}-\frac{1}{n}\big)^{k-\frac{1}{2}}\dy
 \\&\quad +e^{19/300}\pi^{1/4}\sum_{\substack{k\ge K_p}}
 \frac{2^{-k}\tilde R^{2k}(p-1)^kk^{1/4}}{(2k+1) n{k!}}\int_{\frac{\tilde R^2}{\kappa}}^n\frac{1}{\sqrt{y}}\big(\frac{1}{y}-\frac{1}{n}\big)^{k-\frac{1}{2}}\dy\Big\}
 \\&\overset{(a)}{\le}\frac{C_{n,p}}{2} \Big\{\sum_{\substack{1\le k\le K_p-1}}
 \frac{\tilde R^{2k}\|H_{2k}\|_p}{n(2k+1)!}\int_{\frac{\tilde R^2}{\kappa}}^n\frac{1}{\sqrt{y}}\big(\frac{1}{y}-\frac{1}{n}\big)^{k-\frac{1}{2}}\dy
 \\&\quad +e^{19/300}\pi^{1/4}\frac{K_p^{1/4}}{(2K_p+1)}\sum_{\substack{k\ge K_p}}
 \frac{2^{-k}\tilde R^{2k}(p-1)^k}{ n{k!}}\int_{\frac{\tilde R^2}{\kappa}}^n\frac{1}{\sqrt{y}}\big(\frac{1}{y}-\frac{1}{n}\big)^{k-\frac{1}{2}}\dy\Big\}
  \\&\overset{(b)}{\le}\frac{C_{n,p}}{2} \Big\{\sum_{\substack{1\le k\le K_p-1}}
 \frac{\tilde R^{2k}}{n}\Big(\frac{\|H_{2k}\|_p}{(2k+1)!}-\frac{K_p^{1/4}2^{-k}(p-1)^ke^{19/300}\pi^{1/4}}{(2K_p+1)k!}\Big)\int_{\frac{\tilde R^2}{\kappa}}^n\frac{1}{\sqrt{y}}\big(\frac{1}{y}-\frac{1}{n}\big)^{k-\frac{1}{2}}\dy
 \\&\quad +e^{19/300}\pi^{1/4}\frac{K_p^{1/4}}{(2K_p+1)n}
\int_{\frac{\tilde R^2}{\kappa}}^n\frac{1}{\sqrt{1-\frac{y}{n}}}\big[e^{\frac{1}{2}(p-1)\tilde R^2\big(\frac{1}{y}-\frac{1}{n}\big)}-1\big]\dy\Big\}
\end{talign}
where (a) follows from the fact that $x\rightarrow\frac{x^{1/4}}{(2x+1)}$ is decreasing and (b) from the fact that \begin{talign}e^{\frac{1}{2}(p-1)\tilde R^2(\frac{1}{y}-\frac{1}{n})}-1=\sum_{k=1}^{K_p-1}\frac{(\tilde R^2(p-1))^k2^{-k} }{k!}\big(\frac{1}{y}-\frac{1}{n}\big)^k+\sum_{k=K_p}^{\infty}\frac{(\tilde R^2(p-1))^k2^{-k} }{k!}\big(\frac{1}{y}-\frac{1}{n}\big)^k.\end{talign}
   \begin{proofaside}Since  \begin{talign}
       &\sum_{\substack{1\le k\le K-1}}
 \frac{\tilde R^{2k}}{n}\Big(\frac{\|H_{2k}\|_p}{(2k+1)!}-\frac{K^{1/4}2^{-k}(p-1)^ke^{19/300}\pi^{1/4}}{(2K+1)k!}\Big)\int_{\frac{\tilde R^2}{\kappa}}^n\frac{1}{\sqrt{y}}\big(\frac{1}{y}-\frac{1}{n}\big)^{k-\frac{1}{2}}\dy
 \\&\quad +e^{19/300}\pi^{1/4}\frac{K^{1/4}}{(2K+1)n}
\int_{\frac{\tilde R^2}{\kappa}}^n\frac{1}{\sqrt{1-\frac{y}{n}}}\big[e^{\frac{1}{2}(p-1)\tilde R^2\big(\frac{1}{y}-\frac{1}{n}\big)}-1\big]\dy\Big\}
   \end{talign} is decreasing in $K$, its largest value is attained for $K=1\le K_p$. Moreover, for all $x\le \sqrt{1-\frac{\tilde R^2}{n\kappa}}$, we have $\frac{1}{(1-x^2)^k}\le \big(\frac{n\kappa}{\tilde R^2}\big)^{k-1}\frac{1}{1-x^2}$. Hence a change of variables implies
{
\begin{talign}&\frac{C_{n,p}}{2} \Big\{\sum_{\substack{1\le k\le K_p-1}}
 \frac{\tilde R^{2k}}{n}\Big(\frac{\|H_{2k}\|_p}{(2k+1)!}-\frac{K_p^{1/4}2^{-k}(p-1)^ke^{19/300}\pi^{1/4}}{(2K_p+1)k!}\Big)\int_{\frac{\tilde R^2}{\kappa}}^n\frac{1}{\sqrt{y}}\big(\frac{1}{y}-\frac{1}{n}\big)^{k-\frac{1}{2}}\dy
 \\&\quad +e^{19/300}\pi^{1/4}\frac{K_p^{1/4}}{(2K_p+1)n}
\int_{\frac{\tilde R^2}{\kappa}}^n\frac{1}{\sqrt{1-\frac{y}{n}}}\big[e^{\frac{1}{2}(p-1)\tilde R^2\big(\frac{1}{y}-\frac{1}{n}\big)}-1\big]\dy\Big\}
 \\&\le C_{n,p} e^{19/300}\pi^{1/4}\frac{\tilde R^2}{3n\kappa}\sum_{\substack{k\ge 1}}
 \int^{\sqrt{1-\frac{\tilde R^2}{n\kappa}}}_0\frac{x^{2k}(p-1)^k\kappa^k}{2^kk!({1-x^2})}\dx
  \\&\le C_{n,p}\frac{\tilde R^2}{3n\kappa}e^{19/300}\pi^{1/4}
 \int^{\sqrt{1-\frac{\tilde R^2}{n\kappa}}}_0\frac{[e^{\frac{1}{2}(p-1)x^2\kappa}-1]}{({1-x^2})}\dx
\\&\le C_{n,p}\frac{\tilde R^2}{3n\kappa}e^{19/300}\pi^{1/4}[e^{\frac{1}{2}(p-1)(1-\frac{\tilde R^2}{n\kappa})\kappa}-1]
\log(\frac{1+\mnkappa}{1-\mnkappa}).
\end{talign}}
    \end{proofaside}

Next suppose $k\ge 4$  is even.  Then $\big((X_i-X'_i)^k\big)_{i\ge 1}$ are almost surely nonnegative. Moreover, 
\begin{talign}
\E[(X_i-X'_i)^k]\le 2\frac{\tilde R^{k-2}}{\sqrt{n}^k}
\qtext{and} 
\E[(X_i-X'_i)^{kp}]\le 2\frac{\tilde R^{kp-2}}{\sqrt{n}^{kp}}.
\end{talign} 
Therefore we can invoke a moment inequality for nonnegative random variables (\cref{new_label}) to conclude that
\begin{talign}n\|\E[Y^k\mid S_n]\|_p&\le \|\sum_{i\le n}(X_i-X'_i)^k\|_p
\le  \frac{\tilde R^k}{\sqrt{n}^k}\|\Bin(n, \frac{2}{\tilde R^2})\|_p.
\end{talign}Moreover by the triangle inequality and \cref{dino_veg} the following upper also holds\begin{talign}
  n  \big\|\E[Y^k\mid \tilde S_n]\|_p&\le    \big\|\sum_{i\le n}(X_i-X_i')^k\|_p
  \\&\le n\E[(X_1-X'_1)^k]+ \big\|\sum_{i\le n}(X_i-X_i')^k-\E[(X_1-X'_1)^k]\|_p\\&\le \frac{\tilde R^{k-2}}{\sqrt{n}^{k-2}}(1+\frac{1}{\sqrt{n}}\tilde U_{n,p}\tilde R).
\end{talign} 
Hence we obtain \begin{talign}
    n\|\E[Y^k\mid S_n]\|_p&\le \|\sum_{i\le n}(X_i-X'_i)^k\|_p
\le  \frac{B_{p,n}\tilde R^{k-2}}{\sqrt{n}^{k-2}}
\end{talign}
for $B_{p,n}\defeq \min(\frac{\tilde R^2}{n}\|\Bin(n,\frac{2}{\tilde R^2})\|_p, 1+\frac{1}{\sqrt{n}}\tilde U_{n,p}\tilde R)$.
This gives us the upper estimate  
{\begin{talign}
   \int_{-\frac{1}{2}\log(1-\frac{\tilde R^2}{n\kappa})}^{\infty}\frac{e^{-tk}}{\sqrt{1-e^{-2t}}^{k-1}}n\|\E[Y^k\mid S_n]\|_p\dt
   \le B_{p,n}  \int_{-\frac{1}{2}\log(1-\frac{\tilde R^{2}}{n\kappa})}^{\infty}\frac{e^{-tk}\tilde R^{k-2}}{\sqrt{1-e^{-2t}}^{k-1}\sqrt{n}^{k-2}}\dt.
\end{talign} 
To bound $(a_{3,2})$,  it remains to bound
     \begin{talign}&
  \sum_{\substack{k\ge 4:\\ k~\textup{is even}}}\frac{\|H_{k-1}\|_p}{k!} \int_{-\frac{1}{2}\log(1-\frac{\tilde R^{2}}{n\kappa})}^{\infty}\frac{e^{-tk}\tilde R^{k-2}}{\sqrt{1-e^{-2t}}^{k-1}\sqrt{n}^{k-2}}\dt
  \\&\overset{(a)}{\le} %
  \sum_{\substack{k\ge 3:\\k ~\textup{is~odd}}}\frac{\tilde R^{k-1}\|H_{k}\|_p}{(k+1)!\sqrt{n}^{k-1}}\int^{\sqrt{1-\frac{\tilde R^2}{n\kappa}}}_0\frac{x^{k}}{\sqrt{1-x^2}^{k}}\dx%
    \\&\overset{(b)}{\le} %
   \sum_{\substack{k\ge 1}}\frac{\tilde R^{2k}\|H_{2k+1}\|_p}{(2k+2)!}\int^{\sqrt{1-\frac{\tilde R^2}{n\kappa}}}_0\frac{x^{2k+1}}{\sqrt{1-x^2}^{2k+1}n^k}\dx
      \\&\overset{(c)}{\le} %
  \frac{1}{2}\sum_{\substack{k\ge 1}}\frac{\tilde R^{2k}\|H_{2k+1}\|_p}{(2k+2)!\sqrt{n}}\int_{{\frac{\tilde R^2}{\kappa}}}^n\frac{1}{\sqrt{y}}\Big(\frac{1}{y}-\frac{1}{n}\Big)^kdy
\end{talign}
where (a) and (c) are obtained by a change of variable, and (b) is a consequence of the fact that every odd number can be written as $2m+1$ for an $m\ge 1$. 
To upper bound this quantity we will again employ a Hermite polynomial moment bound (\cref{mal_aux_ventre}),  $\|H_{k-1}\|_p\le \sqrt{p-1}^{k-1}\sqrt{k-1!}$, and use Stirling's approximation to deduce 
that, for all $m\in \mathbb{N}\setminus\{0\}$, \begin{talign}\begin{split}\label{fache}
    \sqrt{(2m+1)!}= \sqrt{2m+1}\sqrt{2m!}\ge  \sqrt{2m+1}e^{-19/300} 2^m m! / (\pi m)^{1/4}. 
\end{split}\end{talign}
Hence for any $K_p\in \mathbb{N}_+$ we obtain that      \begin{talign}(a_{3,2})=&
 \sum_{\substack{k\ge 4:\\ k~\textup{is even}}} \frac{\|H_{k-1}\|_p}{k!}\int_{-\frac{1}{2}\log(1-\frac{\tilde R^2}{n\kappa})}^{\infty}\frac{e^{-tk}}{\sqrt{1-e^{-2t}}^{k-1}}n\|\E\big((Y_t)^k|S_n\big)\|_p\dt
\\&\overset{}{\le} %
   \frac{B_{p,n}}{2}\sum_{\substack{K_p>k\ge 1}}\frac{\tilde R^{2k}\|H_{2k+1}\|_p}{(2k+2)!\sqrt{n}}\int_{{\frac{\tilde R^2}{\kappa}}}^n\frac{1}{\sqrt{y}}\Big(\frac{1}{y}-\frac{1}{n}\Big)^kdy
   \\&\quad+ \frac{B_{p,n}}{4}\sum_{\substack{K_p\le k}}\frac{\tilde R^{2k}(p-1)^{k+1/2}}{(k+1)\sqrt{2k+1!}\sqrt{n}}\int_{{\frac{\tilde R^2}{\kappa}}}^n\frac{1}{\sqrt{y}}\Big(\frac{1}{y}-\frac{1}{n}\Big)^kdy
      \\&\overset{}{\le} %
  \frac{B_{p,n}}{2} \sum_{\substack{K_p>k\ge 1}}\frac{\tilde R^{2k}}{\sqrt{n}}\Big(\frac{\|H_{2k+1}\|_p}{(2k+2)!}-
   \frac{2^{-k}e^{19/300}\pi^{1/4}K_p^{1/4}\sqrt{p-1}^{2k+1}}{2(K_p+1)\sqrt{2K_p+1}k!}\Big)\int_{{\frac{\tilde R^2}{\kappa}}}^n\frac{1}{\sqrt{y}}\Big(\frac{1}{y}-\frac{1}{n}\Big)^kdy
   \\&\quad+\frac{B_{p,n}}{4}e^{19/300}\pi^{1/4} \frac{K_p^{1/4}\sqrt{p-1}}{(K_p+1)\sqrt{2K_p+1}}\frac{1}{\sqrt{n}}\int_{{\frac{\tilde R^2}{\kappa}}}^n\frac{1}{\sqrt{y}}\big(e^{\frac{1}{2}(p-1)\tilde R^2(\frac{1}{y}-\frac{1}{n})}-1\big)dy.
\end{talign}}

\begin{proofaside} 
   Since  \begin{talign}
       & \sum_{\substack{K>k\ge 1}}\frac{\tilde R^{2k}}{\sqrt{n}}\Big(\frac{\|H_{2k+1}\|_p}{(2k+2)!}-
   \frac{2^{-k}e^{19/300}\pi^{1/4}K^{1/4}\sqrt{p-1}^{2k+1}}{2(K_p+1)\sqrt{2K+1}k!}\Big)\int_{{\frac{\tilde R^2}{\kappa}}}^n\frac{1}{\sqrt{y}}\Big(\frac{1}{y}-\frac{1}{n}\Big)^kdy
   \\&\quad+e^{19/300}\pi^{1/4} \frac{K^{1/4}\sqrt{p-1}}{2(K+1)\sqrt{2K+1}}\frac{1}{\sqrt{n}}\int_{{\frac{\tilde R^2}{\kappa}}}^n\frac{1}{\sqrt{y}}\big(e^{\frac{1}{2}(p-1)\tilde R^2(\frac{1}{y}-\frac{1}{n})}-1\big)dy
   \end{talign} 
   is decreasing in $K$, its largest value is attained for $K=1\le K_p$. Moreover for all $x\le \sqrt{1-\frac{\tilde R^2}{n\kappa}}$ we have $\frac{1}{\sqrt{1-x^2}^{2k+1}}\le \big(\frac{n\kappa}{\tilde R^2}\big)^{k-1}\frac{1}{\sqrt{1-x^2}^3}$. Hence a change of variables implies 
\begin{talign}&
  \frac{B_{p,n}}{2} \sum_{\substack{K_p>k\ge 1}}\frac{\tilde R^{2k}}{\sqrt{n}}\Big(\frac{\|H_{2k+1}\|_p}{(2k+2)!}-
   \frac{2^{-k}e^{19/300}\pi^{1/4}K_p^{1/4}\sqrt{p-1}^{2k+1}}{2(K_p+1)\sqrt{2K_p+1}k!}\Big)\int_{{\frac{\tilde R^2}{\kappa}}}^n\frac{1}{\sqrt{y}}\Big(\frac{1}{y}-\frac{1}{n}\Big)^kdy
   \\&\quad+\frac{B_{p,n}}{4}e^{19/300}\pi^{1/4} \frac{K_p^{1/4}\sqrt{p-1}}{(K_p+1)\sqrt{2K_p+1}}\frac{1}{\sqrt{n}}\int_{{\frac{\tilde R^2}{\kappa}}}^n\frac{1}{\sqrt{y}}\big(e^{\frac{1}{2}(p-1)\tilde R^2(\frac{1}{y}-\frac{1}{n})}-1\big)dy
        \\&\le B_{p,n} \sum_{\substack{k\ge 1}} \frac{\|H_{2k+1}\|_p}{(2k+2)!} \int_0^{\sqrt{1-\frac{\tilde R^2}{n\kappa}}}\frac{x^{2k+1}\tilde R^{2}\sqrt{\kappa}^{2k-2}}{\sqrt{1-x^2}^{3}n}\dx
                \\&\le\frac{\tilde R^{2}\big(1+\frac{\tilde R \tilde U_{n,p}}{\sqrt{n}}\big)}{n\kappa}\frac{\pi^{1/4}e^{19/300}\sqrt{p-1}}{4\sqrt{3}}\big[e^{\frac{1}{2}(p-1)\kappa \big(1-\frac{\tilde R^2}{n\kappa}\big)}-1\big] \int_0^{\sqrt{1-\frac{\tilde R^2}{n\kappa}}}\frac{x}{\sqrt{1-x^2}^{3}}\dx
         \\&\le \frac{\tilde R\big(1+\frac{\tilde R \tilde U_{n,p}}{\sqrt{n}}\big)}{\sqrt{n\kappa}}\frac{\pi^{1/4}e^{19/300}\sqrt{p-1}}{4\sqrt{3}}\big[e^{\frac{1}{2}(p-1)\kappa \big(1-\frac{\tilde R^2}{n\kappa}\big)}-1\big].
\end{talign}
\end{proofaside}
Therefore, to conclude,  we obtain  
\begin{talign}&(a_3)\le (b_{3,1}^{\kappa,p,\tilde R})\\&\defeq      \frac{B_{p,n}}{2} \Big\{\sum_{\substack{K_p>k\ge 1}}\frac{R^{2k}}{\sqrt{n}}\Big(\frac{\|H_{2k+1}\|_p}{(2k+2)!}-
   \frac{2^{-k}e^{19/300}\pi^{1/4}K_p^{1/4}\sqrt{p-1}^{2k+1}}{2(K_p+1)\sqrt{2K_p+1}k!}\Big)\int_{{\frac{\tilde R^2}{\kappa}}}^n\frac{1}{\sqrt{y}}\Big(\frac{1}{y}-\frac{1}{n}\Big)^k\dy
   \\\label{fr_al_3}&\quad+\frac{B_{p,n}}{4}e^{19/300}\pi^{1/4} \frac{K_p^{1/4}\sqrt{p-1}}{(K_p+1)\sqrt{2K_p+1}}\frac{1}{\sqrt{n}}\int_{{\frac{\tilde R^2}{\kappa}}}^n\frac{1}{\sqrt{y}}\big(e^{\frac{1}{2}(p-1)\tilde R^2(\frac{1}{y}-\frac{1}{n})}-1\big)\dy\Big\}
\\&+\frac{C_{n,p}}{2}\Big\{\sum_{\substack{1\le k\le K_p-1}}
 \frac{\tilde R^{2k}}{n}\Big(\frac{\|H_{2k}\|_p}{(2k+1)!}-\frac{K_p^{1/4}2^{-k}(p-1)^ke^{19/300}\pi^{1/4}}{(2K_p+1)k!}\Big)\int_{\frac{\tilde R^2}{\kappa}}^n\frac{1}{\sqrt{y}}\big(\frac{1}{y}-\frac{1}{n}\big)^{k-\frac{1}{2}}\dy
 \\&\quad +e^{19/300}\pi^{1/4}\frac{K_p^{1/4}}{(2K_p+1)n}
\int_{\frac{1}{n}}^{\frac{\kappa}{\tilde R^2}}\frac{1}{y^{3/2}\sqrt{y-\frac{1}{n}}}\big[e^{\frac{1}{2}(p-1)\tilde R^2\big(y-\frac{1}{n}\big)}-1\big]\dy\Big\}.
\end{talign}
\begin{proofaside}
    Moreover, $(b_{3,1}^{\kappa,p,\tilde R})$ can be further upper-bounded by 
    \begin{talign}\label{fr_al_3b}(b_{3,1}^{\kappa,p,\tilde R})&\le (b_{3,2}^{\kappa, p,\tilde R})\defeq
\frac{\tilde R\big(1+\frac{\tilde R \tilde U_{n,p}}{\sqrt{n}}\big)}{\sqrt{n\kappa}}\frac{\pi^{1/4}e^{19/300}\sqrt{p-1}}{4\sqrt{3}}\big[e^{\frac{1}{2}(p-1)\kappa \big(1-\frac{\tilde R^2}{n\kappa}\big)}-1\big]
\\&+ C_{n,p}\frac{\tilde R^2}{3n\kappa}e^{19/300}\pi^{1/4}[e^{\frac{1}{2}(p-1)(1-\frac{\tilde R^2}{n\kappa})\kappa}-1]
\log(\frac{1+\mnkappa}{1-\mnkappa}).
    \end{talign}
\end{proofaside}  

%% file: arxiv-sections/additional-lemmas.tex
\section{Additional Lemmas}
\begin{lemma}[Hermite polynomial moment bound {\citep[Lem.~3]{bonis2020stein}}]\label{mal_aux_ventre}Let $h_k(x)\defeq e^{x^2/2}\frac{\partial^k}{\partial x}e^{-x^2/2}$
and $H_k\defeq h_k(Z)$ for $Z\sim\Gsn(0,1)$. Then
the following holds for all $k,p\in\mathbb{N}$:
\begin{talign}\|H_k\|_p\le   \sqrt{k!}\sqrt{p-1}^k .\end{talign}%
\end{lemma}

\begin{lemma}[Marcinkiewicz-Zygmund inequality {\citep[Thm.~2.1]{rio2009moment}}]\label{dino_veg} 
Suppose $(\tilde X_i)_{i\ge1}$ are centered \iid observations admitting a finite $p$-th absolute moment for some $p\ge 2$.  Then
 \begin{talign}\|\frac{1}{\sqrt{n}}\sum_{i\le n}\tilde X_i\|_p\le \sqrt{p-1}\|\tilde X_1\|_p.\end{talign}
\end{lemma}
\begin{lemma}[Improved Marcinkiewicz-Zygmund inequality for symmetric random variables]\label{dino_veg_symm} 
Suppose $(\tilde X_i)_{i\ge1}$ are symmetric centered \iid observations admitting a finite $p$-th absolute moment for some $p\ge 2$.  Then
 \begin{talign}\|\frac{1}{\sqrt{n}}\sum_{i\le n}\tilde X_i\|_p\le\sqrt{2}\Big(\frac{\Gamma(\frac{p+1}{2})}{\sqrt{\pi}}\Big)^{1/p}\sqrt{n}\|\tilde X_1\|_p.\end{talign}
If $p\ge 4$, $\E[\tilde X_i^2]\le \tilde\sigma^2$, and  $\E[|\tilde X_i|^p]\le b_p$ , we also have 
       \begin{talign}
    \|\sum_{i\le n} \tilde X_i\|_p\le     \sqrt{2}\Big(\frac{\Gamma(\frac{p+1}{2})}{\sqrt{\pi}}\Big)^{1/p}(\frac{b_p}{\tilde\sigma^2})^{1/(p-2)}\sqrt{\|\Bin(n,(\frac{\tilde\sigma^p}{b_p})^{2/(p-2)})\|_{p/2}}.%
    \end{talign}
\end{lemma}
\begin{proof} As the random variables $(\tilde X_i)_{i\geq1}$ are symmetric, we know that \begin{talign}\sum_{i=1}^n\tilde X_i\overset{{d}}{=}\sum_{i=1}^n \epsilon_i|\tilde X_i|,\end{talign}
where $(\epsilon_i)_{i\geq 1}\distiid\Unif(\{-1,1\}).$
    According to \citet[Thm.~B]{haagerup1981best} we have
    \begin{talign}
        \E[|\sum_{i=1}^n \epsilon_i|\tilde X_i||^p\mid(\tilde X_i)_{i=1}^n]\le 2^{p/2}\frac{\Gamma(\frac{p+1}{2})}{\sqrt{\pi}}(\sum_{i=1}^n\tilde X_i^2)^{p/2}.
    \end{talign}
    The tower property and Jensen's inequality therefore imply that 
    \begin{talign}\label{eq:symm-moment-bound}
      \|\sum_{i=1}^n \tilde X_i\|_p&\le \sqrt{2}\Big(\frac{\Gamma(\frac{p+1}{2})}{\sqrt{\pi}}\Big)^{1/p}(\E[(\sum_{i\le n}\tilde X_i^2)^{p/2}])^{1/p}
      \le\sqrt{2}\Big(\frac{\Gamma(\frac{p+1}{2})}{\sqrt{\pi}}\Big)^{1/p}\sqrt{n}\|\tilde X_1\|_p.
    \end{talign}
    Now suppose $p\ge 4$, $\E[\tilde X_i^2]\le \tilde\sigma^2$, and  $\E[\tilde X_i^p]\le b_p$. Then, by a moment inequality for nonnegative random variables (\cref{new_label}), we obtain 
     \begin{talign}
        \norm{\sum_{i=1}^n\tilde X_i^2}_{p/2}
        \le
        (\frac{b_p}{\tilde\sigma^2})^{2/(p-2)}\|\Bin(n,(\frac{\tilde\sigma^p}{b_p})^{2/(p-2)})\|_{p/2}.
    \end{talign}
    Combining this with the inequality \cref{eq:symm-moment-bound} yields the advertised conclusion.
\end{proof}
\begin{lemma}[Moment inequality for nonnegative random variables]
    \label{new_label}
    Let $(\tilde X_i)_{i\geq 1}$ be a sequence of \iid random variables that are almost surely nonnegative. 
    If, for some $p\ge 2$, $\E(\tilde X_1)\le a$ and $\E(\tilde X_i^p)\le b$ for $a,b>0$, then
    \begin{talign}\|\sum_{i=1}^{n}\tilde X_i\|_p  \le (\frac{b}{a})^{1/(p-1)}\|\Bin(n, (\frac{a^p}{b})^{\frac{1}{p-1}})\|_p\end{talign}
\end{lemma}
\begin{proof}
Let $(V_i)_{i\ge 1}$ be an \iid sequence of Bernouilli random variables with 
\begin{talign}
\P(V_i=0)=1-(\frac{a^p}{b})^{\frac{1}{p-1}}
\qtext{and}
\P(V_i=1)=(\frac{a^p}{b})^{\frac{1}{p-1}}.
\end{talign}
    Then, by \citet[Thm.~2]{ibragimov2001best},
    \begin{talign}
    \E[(\sum_{i=1}^{n}\tilde X_i)^p]\le \E[(\sum_{i=1}^{n}(\frac{b}{a})^{1/(p-1)}V_i)^p]
   \le  (\frac{b}{a})^{p/(p-1)}\E[(\sum_{i=1}^{n}V_i)^p].
    \end{talign}
\end{proof}
\begin{lemma}[Rosenthal's inequality with explicit constants]\label{dino_veg2} \quad Let $(\tilde X_i)_{i\geq 1}$ be a sequence of centered \iid observations. If $\|\tilde X_1\|_p<\infty$ for some $p\ge 2$, then 
 \begin{talign}\|\frac{1}{\sqrt{n}}\sum_{i\le n}\tilde X_i\|_p\le (\frac{p}{2}+1)n^{1/p-1/2}\|\tilde X_1\|_p+2^{1/p}\sqrt{p/2+1}e^{\frac{1}{2}+\frac{1}{p}}\|\tilde X_1\|_2.\end{talign}
\end{lemma}
\begin{proof}
    According to \citet[Thm.~2]{nagaev1978some} we have
    \begin{talign}
        \| \frac{1}{\sqrt{n}}\sum_{i\le n} \tilde X_i\|_p^p 
\le \inf_{c > \frac{p}{2}}  c^p n^{1-\frac{p}{2}} \| \tilde X_1 ||_p^p + p c^{p/2} e^c B(\frac{p}{2}, c-\frac{p}{2})\| \tilde X_1 \|_2^p,
    \end{talign} where $B(\cdot,\cdot)$ is the Beta function.
    The choice $c=\frac{p}{2}+1$ yields
    \begin{talign}
        \| \frac{1}{\sqrt{n}}\sum_{i\le n} \tilde X_i\|_p^p 
&\le (\frac{p}{2}+1)^p n^{1-\frac{p}{2}} \| \tilde X_1 ||_p^p + p (\frac{p}{2}+1)^{p/2} e^{\frac{p}{2}+1} B(\frac{p}{2}, 1)\| \tilde X_1 \|_2^p
\\&= (\frac{p}{2}+1)^p n^{1-\frac{p}{2}} || \tilde X_1 ||_p^p + 2 (\frac{p}{2}+1)^{\frac{p}{2}} e^{\frac{p}{2}+1}  || \tilde X_1 ||_2^p .
    \end{talign}
    The subadditivity of the $p$-th root now implies the result.
\end{proof}
 
\begin{lemma}[Kolmogorov-Smirnov quantile bound]\label{rm_sigma}
    Let $(X_i)_{i\geq 1}$ be a sequence of \iid random variables taking values in $[0,1]$ and, for a confidence level $\alpha>0$, let 
    $\ksqone[\alpha]$ and $\ksq[\alpha]$ respectively be the $1-\alpha$ quantiles of the one-sided and two-sided Kolmogorov-Smirnov (KS) distribution with parameter $n$. Then, 
    \begin{align}
        \P\Big(\sup_{k\in\naturals}\textstyle\frac{1}{n}\sum_{i\le n}X_i^k-\E[X_1^k]\ge \ksqone[\alpha] \Big) \vee
        \P\Big(\displaystyle\sup_{k\in\naturals}|\textstyle\frac{1}{n}\sum_{i\le n}X_i^k-\E[X_1^k]|\ge \ksq[\alpha] \Big)\le \alpha.
    \end{align}
\end{lemma}
\begin{proof}
    We begin by proving the two-sided statement.
    Let $\mu$ be the distribution of $X_1$ and $\hat{\mu}$ the empirical distribution of $(X_i)_{i=1}^n$.
    \citet[Prop.~3.1]{romano2000finite} showed that
    \begin{talign}
    \sup_{k\in\naturals}|\frac{1}{n}\sum_{i\le n}X_i^k-\E[X_1^k]|\le \kd(\muhat,\mu)
    \defeq
    \sup_{x\in[0,1]} |\P(X_1 \leq x) - \P_{X\sim\muhat}(X \leq x)|
    \end{talign} 
    for $\kd(\muhat,\mu)$ the two-sided Kolmogorov distance between $\muhat$ and $\mu$. 
    Now let $\ks(n)$ denote the two-sided Kolmogorov-Smirnov distribution with parameter $n$.
    If $X_1$ is continuous, then $\kd(\muhat,\mu) \sim \ks(n)$ \citep[Thm.~1]{feller1948kolmogorov} and hence
    \begin{talign}\label{eq:two-sided-ksq}
        \P(\sup_{k\in\naturals}|\frac{1}{n}\sum_{i\le n}X_i^k-\E[X_1^k]|\ge q^{\rm{KS}}_{n}(\alpha) )\le \P(\kd(\muhat,\mu)\ge q^{\rm{KS}}_{n}(\alpha))\le \alpha.
    \end{talign}
    If $X_1$ is not continuous, fix any $\epsilon > 0$, and define 
    \begin{talign}X_i^{\epsilon}\defeq \frac{1}{1+\epsilon}(X_i+\epsilon U_i)
    \end{talign} 
    for $(U_i)_{i\geq 1}$  an \iid sequence of uniform random variables in $[0,1]$. 
    Since the $(X_i^{\epsilon})_{i\geq 1}$ are continuous \iid random variables on $[0,1]$, we have
    \begin{talign}
    \P(\sup_{k\in\naturals}|\frac{1}{n}\sum_{i\le n}(X_i^{\epsilon})^k-\E((X_1^{\epsilon})^k)|\ge q^{\rm{KS}}_{n}(\alpha))\le \alpha.\end{talign}
    As this holds for any arbitrary choice of $\epsilon>0$, the result \cref{eq:two-sided-ksq} holds.

    A nearly identical proof establishes the one-sided result, since, using integration by parts,
    \begin{talign}
    \frac{1}{n}\sum_{i\le n}X_i^k-\E[X_1^k]
        &= 
    \int_0^1 
        x^k d(\muhat-\mu)(x)
        =
    \int_0^1
        k x^{k-1} (\P_{X\sim\muhat}(X >x)-\P(X_1 > x))dx \\
        &\leq
        \kdone(\mu, \muhat)
    \int_0^1
        k x^{k-1} dx 
        =
    \kdone(\mu, \muhat)
    \end{talign}
    for $\kdone(\mu, \muhat) \defeq
    \sup_{x\in[0,1]} \P(X_1 \leq x) -\P_{X\sim\muhat}(X \leq x)$, the one-sided Kolmogorov distance.
\end{proof}